\documentclass[12pt,a4paper]{amsart}  

\usepackage{amssymb}
\RequirePackage[raiselinks=false,colorlinks=true,citecolor=blue,urlcolor=black,linkcolor=black,bookmarksopen=true]{hyperref}
\newcommand{\arxiv}[1]{\href{http://arxiv.org/pdf/#1}{arXiv:#1}}
\usepackage{eucal}

\usepackage[v2,all]{xy}

\newcommand{\dlpullback}[1][dl]{\save*!/#1-4ex/#1:(-1,1)@^{|-}\restore}

\newcommand{\drpullback}[1][dr]{\save*!/#1-4ex/#1:(-1,1)@^{|-}\restore}

\DeclareMathAlphabet{\mathbbe}{U}{bbold}{m}{n}
\newcommand{\simplexcategory}{\mathbbe{\Delta}}

\newcommand{\Map}{\operatorname{Map}}
\newcommand{\eq}{\mathrm{eq}}

\newcommand{\un}{\underline}
\newcommand{\relfin}{\mathrm{\,rel.fin.}}
\newcommand{\finsup}{\mathrm{\,fin.sup.}}
\newcommand{\bigcat}{\widehat{\kat{Cat}}}

\newcommand{\Vect}{\kat{Vect}}
\newcommand{\vect}{\kat{vect}}

\providecommand{\norm}[1]{\left| {#1}\right|}
\providecommand{\normnorm}[1]{\left|\!\left| {#1}\right|\!\right|}

\newcommand{\pro}[1]{\underleftarrow{#1}}
\newcommand{\ind}[1]{\underrightarrow{#1}}
\providecommand{\kat}[1]{\text{\textbf{\textsl{#1}}}}

\newcommand{\LIN}{\kat{LIN}}
\newcommand{\lin}{\kat{lin}}
\newcommand{\Lin}{\kat{Lin}}

\newcommand{\upperstar}{^{\raisebox{-0.25ex}[0ex][0ex]{\(\ast\)}}}
\newcommand{\lowershriek}{_!}

\newcommand{\isopil}{\stackrel{\raisebox{0.1ex}[0ex][0ex]{\(\sim\)}}%
			{\raisebox{-0.15ex}[0.28ex]{\(\rightarrow\)}}}
\newcommand{\tensor}{\otimes}
\newcommand{\op}{^{\text{{\rm{op}}}}}

\newcommand{\Set}{\kat{Set}}
\newcommand{\Grpd}{\mathcal{S}}
\newcommand{\grpd}{\mathcal{F}}

\newcommand{\Z}{\mathbb{Z}}
\newcommand{\Q}{\mathbb{Q}}
\newcommand{\CC}{\mathcal{C}}
\newcommand{\DD}{\mathcal{D}}

\newcommand{\PP}{\mathcal{P}}

\def\rTo{\longrightarrow}

\newcommand{\name}[1]{\ulcorner #1\urcorner}
\newcommand{\Fun}{\operatorname{Fun}}
\newcommand{\Aut}{\operatorname{Aut}}
\newcommand{\id}{\operatorname{id}}
\newcommand{\colim}{\operatornamewithlimits{colim}}

\newtheorem{lemma}{Lemma}[section]
\newtheorem{prop}[lemma]{Proposition}

\newtheorem{theorem}[lemma]{Theorem}
\newtheorem{cor}[lemma]{Corollary}

\newtheorem{taller}[lemma]{$\!\!$}

\newenvironment{blanko}[1]{\begin{taller}{\normalfont\bfseries #1}\normalfont}{\end{taller}}
\thanks{%
  The first author 
  was partially supported by grants 
  MTM2012-38122-C03-01,  
  2014-SGR-634,          
  MTM2013-42178-P,       
  MTM2015-69135-P, and    
  MTM2016-76453-C2-2-P,  
  the second author 
  by 
  MTM2013-42293-P and 
  MTM2016-80439-P
  and the third author   
  by
  MTM2013-42178-P and
  MTM2016-76453-C2-2-P}

\author{Imma G\'alvez-Carrillo}
\address{Departament de Matem\`atiques
      \\
      Universitat Polit\`ecnica de Catalunya
      \\
      Escola d'Enginyeria de Barcelona Est (EEBE)
      \\
      Carrer Eduard Maristany 10-14
      \\
      08019 Barcelona\\Spain
}
\email{m.immaculada.galvez@upc.edu}

\author{Joachim Kock}
\address{Departament de Matem\`atiques
       \\Universitat Aut\`onoma de Barcelona
       \\08193 Bellaterra (Barcelona), Spain}
\email{kock@mat.uab.cat}

\author{Andrew Tonks}
\address{Department of Mathematics\\ 
University of Leicester\\ 
University Road\\ 
Leicester LE1 7RH, UK}
\email{apt12@le.ac.uk}

\title{Homotopy linear algebra}
\date{}    

\begin{document}
\begin{abstract}
  By homotopy linear algebra we mean the study of linear functors between slices
  of the $\infty$-category of $\infty$-groupoids, subject to certain finiteness
  conditions.  After some standard definitions and results, we assemble said
  slices into $\infty$-categories to model the duality between vector spaces and
  profinite-dimensional vector spaces, and set up a global notion of homotopy
  cardinality \`a la Baez--Hoffnung--Walker compatible with this duality.  We
  needed these results to support our work on incidence algebras and M\"obius
  inversion over $\infty$-groupoids; we hope that they can also be of independent
  interest.
\end{abstract}

\subjclass[2010]{18-XX, 55Pxx, 46A20, 15Axx, 37F20}


\maketitle

\small

\tableofcontents

\normalsize

\setcounter{section}{-1}

\addtocontents{toc}{\protect\setcounter{tocdepth}{1}}

\section{Introduction}

\begin{blanko}{Vector spaces and spans (Yoneda and B\'enabou).}
  It has been known since the early days of category 
  theory~\cite{Yoneda}, \cite{Benabou:bicat}
  that the category of sets and spans behaves a lot like the category of vector
  spaces.  Its objects are sets; morphisms from $S$ to $T$ are spans
  $$
  S \leftarrow M \to T ,
  $$
  and composition is given by pullback.
  The pullback formula for composition can in fact be written as a
  matrix multiplication.

  Further vector-space flavour can be brought out by a slight reinterpretation of
  the objects in the span category.  Since the sets $S$ and $T$ 
  index the `columns'
  and `rows' in the `matrix' $M$, they play the role of 
  bases of vector spaces. 
  The vector space is the
  scalar-multiplication-and-sum completion of the basis set $S$. 
  The corresponding categorical construction is the slice category $\Set_{/S}$
  whose objects are maps $V \to S$.  The fibre $V_s$ over an element $s\in S$
  plays the role of the coefficient of the basis element $s$ in a linear
  combination.  The way a span acts on a `vector' $V\to S$ is now just a special
  case of matrix multiplication: by pullback, and then composition.  These two
  operations are the most basic functors between slice categories, and they are
  adjoint: for a given set map $f: S \to T$, pullback along $f$ is denoted
  $f\upperstar $ and composition denoted $f\lowershriek$, and we have the
  adjunction
  $$
  \xymatrix{
  \Set_{/S} \ar[r]<+2mm>^{f\lowershriek}  \ar@{}[r]|\bot
  & \ar[l]<+2mm>^{f\upperstar} \Set_{/T}.
  }
  $$
  The category of sets is locally cartesian closed, hence $f\upperstar$
  preserves all colimits.  In particular it preserves
  `linear combinations', and it is appropriate to call composites of
  lowershrieks and upperstars {\em linear functors}.

  The promised stronger vector-space flavour thus comes from
  considering the category whose objects are
  slice categories and whose morphisms are linear functors.
\end{blanko}

\begin{blanko}{Cardinality.}
  The constructions above work with arbitrary sets, but in order to maintain our
  `linear combination' interpretation of an object $V\to S$ of a slice category
  we must impose certain finiteness conditions.  A linear combination is a
  finite sum of finite scalar multiples of vectors from the (possibly infinite)
  basis; we should thus require $V$ (but not $S$) to be a finite set, and suitable
  finiteness conditions should also be imposed on spans.  These finiteness
  conditions are also needed in order to be able to take cardinality and recover
  results at the level of vector spaces.  The cardinality of $V\to S$ is a
  vector in the vector space spanned by $S$, namely the linear combination
  $\sum_{s\in S}\norm{V_s} \delta_s$, where $\delta_s$ is the basis vector
  indexed by $s$, and $\norm{V_s}$ denotes the usual cardinality of the set
  $V_s$.
\end{blanko}

\begin{blanko}{Objective algebraic combinatorics (Joyal, Lawvere, Lawvere--Menni).}
  Linear algebra with sets and spans is most useful in the coordinatised 
  situation (since the slice categories are born with a `basis'), and in 
  situations where the coefficients are natural numbers.  In practice it is
  therefore mostly algebraic combinatorics that can benefit from {\em objective 
  linear algebra} as we call it, following Lawvere, who has advocated the
  {\em objective method} in combinatorics and number theory.
  
  In a nutshell, algebraic combinatorics is the study of combinatorial objects
  via algebraic structures associated to them.  Most basically these algebraic
  structures are vector spaces.
  Further algebraic structure, such as coalgebra structure, is induced
  from the combinatorics of the objects.  Generating functions and incidence algebras
  are two prime examples of this mechanism.
  While these algebraic techniques are very powerful, for enumerative purposes for example, it is
  also widely acknowledged that bijective proofs represent deeper understanding
  than algebraic identities;  this is one motivation for wishing to objectify
  combinatorics.  A highlight in this respect is Joyal's theory of
  species~\cite{JoyalMR633783}, which reveals the objective
  origin of many operations with power series.  A species is a $\Set$-valued
  functor on the groupoid of finite sets and bijections; the value on an 
  $n$-element set is the objective counterpart to the $n$th coefficient in the
  corresponding exponential generating function.
    
  The present work was motivated by incidence algebras and M\"obius inversion.  A M\"obius
  inversion formula is classically an algebraic identity in the incidence
  algebra (of a locally finite poset, say, or more generally a M\"obius category in
  the sense of Leroux~\cite{Leroux:1975}), hence an equation between two linear
  maps.  So by realising the two linear maps as spans and establishing a
  bijection between the sets representing these spans, a bijective proof can be 
  obtained.   An objective M\"obius inversion principle for M\"obius $1$-categories was
  established by Lawvere and Menni~\cite{LawvereMenniMR2720184}; the 
  $\infty$-version of these results \cite{GKT:DSIAMI-1}, \cite{GKT:DSIAMI-2}, 
  \cite{GKT:MI} required the developments of homotopy linear algebra of the 
  present paper.
\end{blanko}

\begin{blanko}{From sets to groupoids (Baez--Dolan, Baez--Hoffnung--Walker).}
  Baez and Dolan~\cite{Baez-Dolan:finset-feynman} discovered that the theory of
  species can be enhanced by considering groupoid-valued species instead of
  set-valued species.  The reason is that most combinatorial objects have
  symmetries, which are not efficiently handled with $\Set$-coefficients.  Their
  paper~\cite{Baez-Dolan:finset-feynman} illustrated this point by
  showing that
  the exponential generating function corresponding to a species is literally 
  the cardinality of the associated analytic functor, provided this analytic 
  functor is taken with groupoid coefficients rather than set coefficients.
  They also showed
  how the annihilation and creation operators in Fock space can be given
  an objective combinatorial interpretation in this setting.
   
  A subsequent paper by Baez, Hoffnung and
  Walker~\cite{Baez-Hoffnung-Walker:0908.4305} developed in detail the basic
  aspects of linear algebra over groupoids, under the name `groupoidification'.
  One important contribution was to check that the symmetry factors that
  arise behave as expected and cancel out appropriately in the various
  manipulations.  A deeper insight in their paper is to clarify the notion of
  (groupoid) cardinality by deriving all cardinality assignments, one for each
  slice category, from a single global prescription, defined as a functor from
  groupoids and spans to vector spaces.
  This does not work for all groupoids; the ones that admit a cardinality
  are called {\em tame}, a notion akin to square-integrability,  and convergence plays a role. 
  There is a corresponding notion of tame span.
\end{blanko}

\medskip

We now come to the new contributions of the present paper, and as a first
approximation we explain them in terms of three contrasts with 
Baez--Hoffnung--Walker:

\begin{blanko}{From groupoids to $\infty$-groupoids.}
  We work with coefficients in $\infty$-groupoids,
  so as to incorporate more homotopy theory.  The abstraction step from 
  $1$-groupoids to $\infty$-groupoids is actually not so drastic, since the
  theory of $\infty$-categories is now so well developed that one can deal with
  elementary aspects of $\infty$-groupoids with almost the same ease as one
  deals with sets (provided one deals with sets in a good categorical way).
\end{blanko}

\begin{blanko}{Homotopy notions.}  
  In fact, because of the abstract viewpoints forced upon us by the setting of
  $\infty$-groupoids, we are led to some conceptual simplifications,
  valuable
  even when our results are specialised to the $1$-groupoid level.
  The main point is the consistent homotopy approach.  We
  work consistently with homotopy fibres, while in \cite{Baez-Hoffnung-Walker:0908.4305}
  `full' fibres are employed.  We
  exploit homotopy sums, where \cite{Baez-Hoffnung-Walker:0908.4305} spells out 
  the formulae in ordinary sums,
  at the price of carrying around symmetry factors. 
  Just as an ordinary sum is a colimit indexed over a set, a homotopy sum is a 
  colimit indexed over an $\infty$-groupoid.
  The advantage of working with homotopy fibres and homotopy sums is that
  homotopy sums (an example of $f\lowershriek$) are left adjoint to homotopy
  fibres (an example of $f\upperstar$) in exactly the same way as, over sets,
  sums are left adjoint to fibres.  As a consequence, with the correct
  notation, no symmetry factors appear --- they are absorbed into the 
  formalism and take
  care of themselves.  (See \cite{GalvezCarrillo-Kock-Tonks:1207.6404} for
  efficient exploitation of this viewpoint.)
\end{blanko}

\begin{blanko}{Finiteness conditions and duality issues.}
  A more substantial difference to the Baez--Hoffnung--Walker approach concerns
  the finiteness conditions.  Their motivating example of Fock space led them to
  the tameness notion which is self-dual: if a span is tame then so is the
  transposed (or adjoint) span (i.e.~the same span read backwards).  
  One may say that they
  model Hilbert spaces rather than plain vector spaces.

  Our motivating examples are incidence coalgebras and incidence algebras; these
  are naturally vector spaces and profinite-dimensional vector spaces, 
  respectively, 
  and a fundamental fact is the classical duality between vector spaces and 
  profinite-dimensional vector spaces.  Recall that if $V$ is a 
  vector space, the linear dual $V\upperstar$ is naturally a
  profinite-dimensional vector space, and that in turn the
  continuous linear dual of $V\upperstar$ is naturally isomorphic
  to $V$. In the fully coordinatised situation characteristic
  of algebraic combinatorics, $S$ is some set of (isoclasses of)
  combinatorial objects, the vector space spanned by $S$ is
  the set of finite linear combinations of elements in $S$, which
  we denote by $\Q_S$, and the dual can naturally be identified 
  with the space of $\Q$-valued functions, $\Q^S$.  (Some further
  background on this duality is reviewed in \ref{vect-rappels} below.)
  
  The appropriate finiteness condition to express these notions is simply
  homotopy finiteness: an $\infty$-groupoid is called homotopy finite,
  or just {\em finite}, when it has finitely many components, all homotopy
  groups are finite, and there is an upper bound on the dimension of nontrivial
  homotopy groups.  A morphism of $\infty$-groupoids is called finite
  when all its fibres are finite. Letting $\grpd$ denote the $\infty$-category
  of finite $\infty$-groupoids, the role of vector spaces is played by
  finite-$\infty$-groupoid slices $\grpd_{/S}$, while the role of
  profinite-dimensional vector spaces is played by finite-presheaf
  $\infty$-categories $\grpd^S$, where in both cases $S$ is only required to be
  locally finite.  Linear maps are given by spans of {\em finite type}, meaning
  $S \stackrel p\leftarrow M \stackrel q\to T$ in which $p$ is a finite map.
  Prolinear maps are given by spans of {\em profinite type}, where instead $q$
  is a finite map.  We set up two $\infty$-categories: the $\infty$-category
$\ind\lin$ whose objects
  are the slices $\grpd_{/S}$ and whose mapping spaces are $\infty$-groupoids of
  finite-type spans, and the $\infty$-category $\pro\lin$ whose objects are
  finite-presheaf $\infty$-categories $\grpd^S$ and whose mapping spaces are
  $\infty$-groupoids of profinite-type spans;
we show that these are dual.
We introduce a global notion of cardinality such that the classical duality becomes the cardinality of the $\ind\lin$ - $\pro\lin$ duality.
\end{blanko}

\bigskip

We proceed to outline the paper, section by section.

\bigskip

  The finiteness conditions are needed to be able to take homotopy cardinality.
  However, as long as we are working at the objective level, it is not necessary
  to impose the finiteness conditions, and in fact, the theory is simpler
  without them.  Furthermore, the notion of homotopy
  cardinality is not the only notion of size: Euler characteristic and various
  multiplicative cohomology theories are other potential alternatives, and it is
  reasonable to expect that the future will reveal more comprehensive and
  unified notions of size and measures.  For these reasons, we begin
  in Section~\ref{sec:LIN} with `linear algebra' without finiteness conditions.

  Let $\Grpd$ denote the $\infty$-category of $\infty$-groupoids.
  We define formally the $\infty$-category $\LIN$, whose objects are slices
  $\Grpd_{/S}$ and whose morphisms are linear functors.  We show that the
  $\infty$-category $\Grpd_{/S}$ is the homotopy-sum completion of $S$, and interpret
  scalar multiplication and homotopy sums as special cases of the lowershriek operation.  The
  canonical basis is given by the `names', functors $\name x: 1 \to S$.  We show
  that linear functors can be presented canonically as spans.  We exploit
  results already proved by Lurie~\cite{Lurie:HA} to establish that $\LIN$ is symmetric monoidal
  closed.  The tensor product is given by
  $$
  \Grpd_{/S} \tensor \Grpd_{/T} = \Grpd_{/S\times T} .
  $$

 \medskip
 
  In Section~\ref{sec:finite} we start getting into finiteness
  conditions.  An $\infty$-groupoid $X$ is {\em locally finite} if at each base
  point $x$ the homotopy groups $\pi_i (X,x)$ are finite for $i\geq1$ and are
  trivial for $i$ sufficiently large.  It is called {\em finite} if furthermore
  it has only finitely many components.  The cardinality of a finite
  $\infty$-groupoid $X$ is defined as
  $$
  \norm{X} := \sum_{x\in \pi_0 X} \prod_{i>0} \norm{\pi_i(X,x)}^{(-1)^i} .
  $$
  We work out the basic properties of this notion, notably how it interacts with
  pullbacks in special cases.  We check that the $\infty$-category $\grpd$
  of finite $\infty$-groupoids is locally cartesian closed.

  \medskip
  
  In Section~\ref{sec:finconslice} we first recall the duality between vector
  spaces and profinite-dimensional vector spaces, on which the $\ind\lin$ -
  $\pro\lin$ duality is modelled.
  
  The basis $S$ is required to be locally finite, in order to have
  pullback stability of finite $\infty$-groupoids over it, but it is
  essential not to require it to be finite, as the vector spaces we 
  wish to model are not finite dimensional.    To the category of vector spaces
  corresponds the $\infty$-category $\grpd_{/S}$ of finite $\infty$-groupoids over $S$.  To the category of
  profinite-dimensional vector spaces corresponds the $\infty$-category $\grpd^S$ of 
  finite-$\infty$-groupoid-valued presheaves.  We also introduce the variants 
  $\grpd^S_{\finsup}$ of presheaves with finite support, and 
  $\Grpd_{/S}^{\relfin}$ of finite maps to $S$; the latter can be thought of 
  as a space of measures on $S$ (in view of \ref{dualduality}).  These two 
  $\infty$-categories are naturally
  equivalent to the previous pair, but live on the opposite side of the 
  duality we are setting up.
  
  We proceed to assemble these collections of finite slices into the
  following $\infty$-categories.
 
  There is an $\infty$-category $\lin$ whose objects are $\infty$-categories of the form 
  $\grpd_{/\alpha}$ where $\alpha$ is a finite $\infty$-groupoid, and with morphisms
  given by
  finite spans $\alpha \leftarrow \mu \to \beta$.  This $\infty$-category corresponds to
  the category $\vect$ of
  finite-dimensional vector spaces. We need infinite indexing,
  so the following two extensions are introduced, referring to a 
  locally finite $\infty$-groupoid $S$.
  There is an $\infty$-category $\ind\lin$ whose objects are $\infty$-categories of the form
  $\grpd_{/S}$,
  and whose morphisms are
  spans of finite type (i.e.~the left leg has finite fibres).  This $\infty$-category 
  corresponds to the category $\ind\vect$ of general vector spaces
  (allowing infinite-dimensional ones).  
  Finally we have the $\infty$-category $\pro\lin$ whose objects are $\infty$-categories of the form
  $\grpd^S$ with $S$ a locally finite $\infty$-groupoid, and whose morphisms are
  spans of profinite type (i.e.~the right leg has finite fibres).  This $\infty$-category 
  corresponds to the category $\pro\vect$ of profinite-dimensional vector spaces.
  
  \medskip
  
  In order actually to define $\lin$, $\ind\lin$ and $\pro\lin$ 
  as $\infty$-categories, in Section~\ref{sec:weird}
  we take an intermediate step up in the realm of presentable $\infty$-categories
  --- so to speak extending scalars from $\grpd$ to $\Grpd$ --- to be able to 
  leverage our work from Section~\ref{sec:LIN}.
  
  So, within the ambient $\infty$-category $\LIN$ we define the following
  subcategories: the $\infty$-category $\Lin$ with objects of the form
  $\Grpd_{/\alpha}$ and morphisms given by finite spans; the $\infty$-category
  $\ind\Lin$ consisting of $\Grpd_{/S}$ and spans of finite type; and the
  $\infty$-category $\pro\Lin$ consisting of $\Grpd^S$ and spans of profinite
  type.
   
  We characterise profinite spans by the following pleasant `analytic'
  continuity condition (\ref{prop:cont}): 
  
  {\em A linear functor $F: \Grpd^{T}\to \Grpd^{S}$ is given by a profinite span
  if and only if for all $\varepsilon \subset S$ there exists $\delta \subset T$
  and a factorisation
  $$\xymatrix{
      \Grpd^{T} \ar[r]\ar[d]_F &  \Grpd^{\delta} \ar[d]^{F_{\delta}} \\
      \Grpd^{S}\ar[r] &  \Grpd^{\varepsilon}
  }$$
  where $\varepsilon$ and $\delta$ denote finite $\infty$-groupoids, 
  and the horizontal maps are the projections of the canonical pro-structures.
}
     
   The three $\infty$-categories constructed with $\Grpd$ coefficients 
   are in fact equivalent to the three $\infty$-categories with $\grpd$ coefficients
   introduced heuristically.
   
   \medskip
   
   In Section~\ref{sec:duality} we establish that the pairing
   $\grpd_{/S} \times \grpd^S \to \grpd$ is perfect.  In Section~\ref{sec:metacard}
   we prove that upon taking cardinality this yields the
   pairing $\Q_{\pi_0 S} \times \Q^{\pi_0 S} \to \Q$.
  To define the cardinality notions, 
 we follow Baez--Hoffnung--Walker~\cite{Baez-Hoffnung-Walker:0908.4305}
and introduce a 
`meta cardinality' functor, which induces cardinality notions in all slices and 
in all presheaf $\infty$-categories.
In our setting, this amounts to a functor
\begin{eqnarray*}
  \normnorm{ \ } : \ind\lin & \longrightarrow & \Vect  \\
  \grpd_{/S} & \longmapsto & \Q_{\pi_0 S}
\end{eqnarray*}
and a dual functor
\begin{eqnarray*}
  \normnorm{ \ } : \pro\lin  & \longrightarrow & \pro\vect  \\
  \grpd^S & \longmapsto & \Q^{\pi_0 S} .
\end{eqnarray*}
For each fixed $\infty$-groupoid $S$, this gives an individual notion of
cardinality $\norm{ \ } : \grpd_{/S} \to \Q_{\pi_0 S}$ (and dually
$\norm { \ } : \grpd^S \to \Q^{\pi_0 S}$), since vectors are just
linear maps from the ground field.

The vector space $\Q_{\pi_0 S}$ is spanned by the elements $\delta_s :=
\norm{\name s}$.  Dually, the profinite-dimensional vector space $\Q^{\pi_0 S}$
is spanned by the characteristic functions $\delta^t
=
\frac{\norm{h^t}}{\norm{\Omega(S,t)}}$
(the cardinality of the representable functors divided by the
cardinality of the loop space).

\bigskip

\begin{blanko}{Related work.}
  Part of the material developed here may be considered either folklore, or
  straightforward generalisations of well-known results in $1$-category theory,
  or special cases of fancier machinery.

  $(\infty,1)$-categories of spans have been studied by many people in
  different contexts and with different goals, e.g.~Lurie~\cite{Lurie:HA},
  Dyckerhoff--Kapranov~\cite{Dyckerhoff-Kapranov:1212.3563} and
  Barwick~\cite{Barwick:1301.4725}.  Lurie~\cite{Lurie:GoodwillieI} studies
  an $(\infty,2)$ version relevant for the present purposes; Dyckerhoff and 
  Kapranov~\cite{Dyckerhoff-Kapranov:1212.3563} study a different 
  $(\infty,2)$-category of spans; and
  Haugseng~\cite{Haugseng:1409.0837}, motivated by topological field
  theory~\cite{Lurie:0905.0465}, studies an $(\infty,n)$-category of iterated
  spans (which for $n=2$ is different from both the previous).
  
  Finally, the theory of slices and linear functors is subsumed into
  the theory of polynomial functors, where a further right adjoint enters the
  picture, the right adjoint to pullback.  The theory of polynomial functors
  over $\infty$-categories is developed in \cite{Gepner-Kock}; see 
  \cite{Gambino-Kock:0906.4931} for the classical case.
\end{blanko}

\noindent {\bf Note.}
  This paper was originally written as an appendix to \cite{GKT:1404.3202},
  to provide precise statements and proofs of the results
  in homotopy linear algebra needed in the theory of decomposition spaces,
  an $\infty$-groupoid setting for
  incidence algebras and M\"obius inversion.
  That manuscript has now been split into smaller papers
  \cite{GKT:DSIAMI-1},
  \cite{GKT:DSIAMI-2}, \cite{GKT:MI}, \cite{GKT:ex}, \cite{GKT:restriction},
  its appendix becoming the present paper.

\bigskip

\noindent
{\bf Acknowledgements.}
We thank Rune Haugseng for very useful feedback on many points in this paper,
and Andr\'e Joyal for the enormous influence he has had on the larger project
this paper is a part of.

\section{Preliminaries on $\infty$-groupoids and $\infty$-categories}

We work with $\infty$-categories, in the sense of Joyal~\cite{Joyal:CRM} and
Lurie~\cite{Lurie:HTT}. We can get away with working model-independently,
since our undertakings are essentially elementary: our objects of study are the
$\infty$-category of $\infty$-groupoids and its slices, and many of the
arguments (for example concerning pullbacks) can be carried out
almost as if we were working with the category of sets --- with a few homotopy
caveats.

In the implementation of $\infty$-categories as quasi-categories, $\infty$-groupoids
are precisely Kan complexes, and serve as a model for topological spaces up to homotopy.
For example, to each object $x$ in an $\infty$-groupoid $X$, there are
associated homotopy groups $\pi_n(X,x)$ for $n>0$; a map $X \to Y$ of
$\infty$-groupoids is an equivalence if and only if it induces a bijection on
the level of $\pi_0$ and isomorphisms on all homotopy groups, and so on.
As is standard, let $\Grpd$ denote the $\infty$-category of $\infty$-groupoids.

The great insight of Joyal~\cite{Joyal:qCat+Kan} was to fit this into a theory
of $\infty$-categories, in which $\infty$-groupoids play the role that {\em
sets} play in category theory.  For example, for any two objects $x,y$ in an
$\infty$-category $\CC$ there is (instead of a hom set) a mapping space
$\Map_\CC(x,y)$ which is an $\infty$-groupoid.  Universal properties, such as
limits, colimits and adjoints can be expressed as equivalences of mapping
spaces.  Presheaves take values in $\infty$-groupoids, and constitute the
colimit completion.

\begin{blanko}{Slices and Beck--Chevalley.}
  Maps of $\infty$-groupoids with codomain $S$ form the objects of a slice
  $\infty$-category $\Grpd_{/S}$, which behaves very much like a slice category
  in ordinary category theory.  (We should mention here that since we work
  model-independently, when we refer to $\Grpd_{/S}$ we refer to an
  $\infty$-category determined up to equivalence.  In contrast,
  \cite{Joyal:CRM} and \cite{Lurie:HTT} often refer to two different
  specific models in the category of simplicial sets with the Joyal model
  structure, which while of course equivalent, have different
  technical advantages.)

  Pullback along a morphism $f: T \to S$
  defines a functor $f\upperstar :\Grpd_{/S} \to \Grpd_{/T}$.  This functor
  is right adjoint to the functor $f\lowershriek:\Grpd_{/T} \to \Grpd_{/S}$ given by
  post-composing with $f$.  
  The following Beck--Chevalley rule (push-pull formula)
  \cite{Gepner-Kock}
  holds for $\infty$-groupoids: given a pullback square
  $$\xymatrix{
  \cdot \drpullback \ar[r]^f \ar[d]_p & \cdot \ar[d]^q \\
  \cdot \ar[r]_g & \cdot}$$
  there is a canonical equivalence of functors 
  \begin{equation}\label{BC}
  p\lowershriek \circ f\upperstar \simeq g\upperstar \circ q\lowershriek .
 \end{equation}
\end{blanko}

\begin{blanko}{Defining $\infty$-categories and sub-$\infty$-categories.}
  In this work we are concerned in particular with defining certain 
  $\infty$-categories, 
  a task often different in nature than that of defining ordinary categories: while
  in ordinary category theory one can define a category by saying what the
  objects and the arrows are (and how they compose), this from-scratch approach
  is more difficult for $\infty$-categories, as one would have to specify the
  simplices in all dimensions and verify the filler conditions (that is,
  describe the $\infty$-category as a quasi-category).  In practice,
  $\infty$-categories are constructed from existing ones by general
  constructions that automatically guarantee that the result is again an
  $\infty$-category, although the construction typically uses universal
  properties in such a way that the resulting $\infty$-category is only defined
  up to equivalence.  To specify a sub-$\infty$-category of an $\infty$-category
  $\CC$, it suffices to specify a subcategory of the homotopy category of $\CC$
  (i.e.~the category whose hom sets are $\pi_0$ of the mapping spaces of $\CC$),
  and then pull back along the components functor.  What this amounts to in
  practice is to specify the objects (closed under equivalences) and specifying
  for each pair of objects $x,y$ a full sub-$\infty$-groupoid of the
  mapping space $\Map_\CC(x,y)$, also closed under equivalences, and closed 
  under composition.
  
  We will use the terms subcategory and subgroupoid rather than the more clumsy 
  `sub-$\infty$-category' and `sub-$\infty$-groupoid'.
\end{blanko}

\begin{blanko}{Fundamental equivalence.}\label{fund-eq}
  Recall that $\Grpd$ is the $\infty$-category of $\infty$-groupoids.
  Fundamental to many constructions and arguments in this work
  is the canonical equivalence
$$
\Grpd_{/S} \simeq \Grpd^S
$$
which is the homotopy version of the equivalence $\Set_{/S} \simeq \Set^S$ (for 
$S$ a set),
expressing the two ways of encoding a family of sets $\{X_s \mid s\in S\}$: 
either regarding the 
members of the family as the fibres of a map $X \to S$, or as a parametrisation
of sets $S \to \Set$.
To an object $X \to S$ one associates the functor $S\op\to\Grpd$ sending
$s\in S$ to the $\infty$-groupoid $X_s$.  The other direction is the Grothendieck construction,
which works as follows:
to any presheaf $F:S\op \to \Grpd$, which sits over the terminal presheaf $*$,
one associates the object $\colim (F)\to\colim(*)$.
It remains to observe that $\colim (*)$ is equivalent to $S$ itself.
More formally, the Grothendieck construction equivalence is a consequence of
a finer result, namely Lurie's straightening theorem
(\cite[Theorem 2.1.2.2]{Lurie:HTT}), as has also been observed in
\cite[Remark 2.6]{ABGHR}.
Lurie constructs 
a Quillen equivalence between the category of right
fibrations over $S$ and the category of (strict) simplicial presheaves on
$\mathfrak{C}[S]$.  
Combining this result with the fact that simplicial presheaves on $\mathfrak{C}[S]$
is a model for the functor $\infty$-category $\Fun(S\op,\Grpd)$ (see \cite{Lurie:HTT}, 
Proposition 5.1.1.1), the Grothendieck 
construction equivalence follows.
\end{blanko}

\section{Homotopy linear algebra without finiteness conditions}
\label{sec:LIN}

In this section we work over $\Grpd$, the 
$\infty$-category of $\infty$-groupoids. 

\begin{blanko}{Scalar multiplication and homotopy sums.}\label{scalar&hosum}
  The  `lowershriek' operation
$$
f\lowershriek:\Grpd_{/I}\to\Grpd_{/J}
$$
along a map $f:I\to J$ has two special cases, which play the role of 
  scalar multiplication (tensoring with an $\infty$-groupoid)
  and vector addition (homotopy sums):
  
The $\infty$-category $\Grpd_{/I}$
is {\em tensored} over $\Grpd$.
Given 
  $g:X \to I$ in $\Grpd_{/I}$
  then for any $S \in \Grpd$ we consider the projection
  $p_S:S\times X\to X$ in $\Grpd_{/X}$ 
  and put
$$
S\otimes g\;\;:=\;\; g\lowershriek (p_S):S\times X\to I
\text{ in }\Grpd_{/I}. 
$$
It also has {\em homotopy sums}, by which we mean colimits 
indexed by an $\infty$-groupoid.
The colimit of a functor $F: B \to \Grpd_{/I}$ is a special case of the 
lowershriek.  Namely, the functor $F$ corresponds by adjunction to an object
$g: X \to B \times I$ in $\Grpd_{/B\times I}$, and we have
$$
\colim(F) = p\lowershriek(g)
$$
where $p: B\times I \to I$ is the projection.
We interpret this as the homotopy sum of the family $g:X\to B\times I$
with members $$g_b:X_b\longrightarrow \{b\}\times I=I,$$
and we denote the homotopy sum by an integral sign:
\begin{equation}\label{fibsum}
\int^{b\in B}g_b \ := \ p\lowershriek g \;\text{ in }\Grpd_{/I}.
\end{equation}
(The use of an integral sign, with superscript, is standard 
notation for colimits that arise as coends~\cite{MacLane:categories}.)
\end{blanko}

\begin{blanko}{Example.}
  With $I=1$, 
this
  gives the important formula
  $$
  \int^{b\in B} X_b = X ,
  $$
  expressing the total space of $X\to B$ as the homotopy sum of its fibres.
\end{blanko}

Using the above, we can define the $B$-indexed {\em linear
combination} of a family of vectors $g:X\to B\times I$ and scalars $f:
S \to B$, 
$$
\int^{b\in B} S_b \tensor g_b = 
p\lowershriek ( g \lowershriek (f'))
\;: S\times_B X\to I
\;\text{ in }\Grpd_{/I},
$$
as illustrated in the first row of the following diagram
\begin{equation}\label{lincomb}
  \vcenter{
      \xymatrix@R0.9pc{ 
            S\times_B X\drpullback  \ar[r]^-{f'
            } \ar[dd] &X \ar[dd]_{qg}\ar[rd]|(0.45){\:\!g\,\!}\ar^-{pg}[rr]
            & & I\\
            & &B\times I\!\!\!\!\!\!\!
            \ar[ru]^(0.38){p\!\!\!\!}\ar[ld]_(0.57){q\!\!\!}\\
            S \ar^{f
            }[r]& B 
      }
  }
\end{equation}
Note that the members of the family $g\lowershriek(f'
)$  are  just $(g\lowershriek(f'
))_b=S_b\otimes g_b$. 

\begin{blanko}{Basis.}
  In $\Grpd_{/S}$, the names $\name s : 1 \to S$ play the role of a basis.
  Every object $X\to S$ can be written uniquely as a linear combination of basis
  elements; or, by allowing repetition of the basis elements instead of scalar 
  multiplication, as a homotopy sum of basis elements:
\end{blanko}

\begin{lemma}\label{BasisRep}
  For any $f:S\to B$  in $\Grpd_{/B}$  we have
  $$
  f = \int^{s\in S} \name{f(s) } = \int^{b\in B} S_b\otimes \name{b} .
  $$
\end{lemma}

\begin{proof}
  The first equality is an example of
  the definition of homotopy sum \eqref{fibsum},
  applied to the family $S \stackrel{(\id, f)} 
  \longrightarrow S \times B$ with 
  members $\name {f(s)}:1=S_{s}\to \{s\}\times B=B$. 
  For the final expression, consider the family
  $g:B \stackrel{(\id, \id)}
  \longrightarrow B \times B$ with members the names $\name b$,
  and the scalars given by $f:S\to B$ itself.
  Then calculating the linear combination  $\int^{b\in B} S_b\otimes \name{b}$
  by \eqref{lincomb} gives just $f$,
  since $pg$ and $qg$ are the identity. 
\end{proof}

The name $\name b : 1 \to B$ corresponds under the
Grothendieck construction to the representable functor
\begin{eqnarray*}
  B & \longrightarrow & \Grpd  \\
  x & \longmapsto & \Map(b,x) .
\end{eqnarray*}
Thus, interpreted in the presheaf category $\Grpd^B$, the
Lemma is the standard result expressing any presheaf as a colimit of representables.

\begin{prop}\label{HoSum}
  $\Grpd_{/S}$ is the homotopy-sum completion of $S$.
  Precisely, for $\CC$ an $\infty$-category admitting
  homotopy sums, precomposition with the Yoneda embedding
  $S \to \Grpd_{/S}$ induces an equivalence of $\infty$-categories
  $$
  \Fun^{\int}(\Grpd_{/S},\CC) \isopil \Fun(S,\CC) ,
  $$
  where the functor category on the left consists of homotopy-sum preserving functors.
\end{prop}

\begin{proof}
  Since every object in $\Grpd_{/S}$ can be written
  as a homotopy sum of names, to preserve
  homotopy sums is equivalent to preserving all colimits, so the natural
  inclusion $\Fun^{\colim}(\Grpd_{/S},\CC) \to \Fun^{\int}(\Grpd_{/S},\CC)$
  is an equivalence.  It is therefore enough to establish the equivalence
  $$\Fun^{\colim}(\Grpd_{/S},\CC) \isopil \Fun(S,\CC).$$
  In the case where $\CC$ is cocomplete, this is true
  since
  $\Grpd_{/S} \simeq \Fun(S\op,\Grpd)$ is the colimit completion of $S$.  The
  proof of this statement (Lurie~\cite{Lurie:HTT}, Theorem~5.1.5.6) goes as
  follows: it is enough to prove that left Kan extension of any functor $S \to
  \CC$ along the Yoneda embedding exists and preserves colimits.  Existence 
  follows from \cite[Lemma~4.3.2.13]{Lurie:HTT}
  since $\CC$ is assumed cocomplete, and the fact
  that left Kan extensions preserve colimits \cite[Lemma~5.1.5.5 (1)]{Lurie:HTT}
  is independent of the cocompleteness of $\CC$.  In our case $\CC$ is not
  assumed to be cocomplete but only to admit homotopy sums.  But since $S$ is
  just an $\infty$-groupoid in our case, 
  this is enough to apply Lemma~4.3.2.13 of \cite{Lurie:HTT} 
  to guarantee the existence of the left Kan extension.
\end{proof}

\begin{blanko}{Linear functors.}\label{linearfunctorsfromspans}
  A span 
$$
I \stackrel p \leftarrow M \stackrel q \to J
$$
defines a {\em linear functor} 
 \begin{equation}\label{linfunctor}
 \Grpd_{/I} \stackrel {p^{\upperstar} } \rTo 
 \Grpd_{/M} \stackrel{q\lowershriek} \rTo \Grpd_{/J} .
 \end{equation}
\end{blanko}

\begin{lemma}\label{linearity}
  Linear functors preserve linear combinations,
$$L\left(\int^{b\in B} S_b\otimes g_b\right)\;=\;\int^{b\in B} S_b\otimes L(g_b).$$
\end{lemma}
\begin{proof}
  This follows from the Beck--Chevalley rule \eqref{BC}, since linear
  combinations (that is, scalar multiplication and homotopy sums) are colimits
  defined using lowershriek operations (see \ref{scalar&hosum}).
\end{proof}

\begin{blanko}{Matrices.}
  Coming back to the span
$$I\stackrel p\longleftarrow M \stackrel q\longrightarrow J$$
and the linear functor
$$
q\lowershriek p\upperstar
:\Grpd_{/I}\longrightarrow\Grpd_{/J},
$$
consider an element $\name i:1\to I$. Then we have, by Lemma~\ref{BasisRep},
$$
q\lowershriek p\upperstar\name i
=(M_i\to J)=
\int^{j\in J}{M_{i,j}}\otimes \name j
$$
$$
\xymatrix{&M_i\drto\dlto\\1\drto_{\name i}&&\dlto^pM\drto_q\\&I&&J}
$$
For  a more general element $f:X\to I$ we have $f=\int^iX_i\otimes \name{i}$ and so by homotopy linearity \ref{linearity}
\begin{align*}
q\lowershriek p\upperstar f&=
\int^{i,j}X_i\otimes M_{i,j}\otimes \name j.
\end{align*}
\end{blanko}

\begin{blanko}{The symmetric monoidal closed $\infty$-category $\kat{Pr}^{\mathrm L}$.}
  There is an $\infty$-category $\kat{Pr}^{\mathrm L}$, defined and studied in
  \cite[Section 5.5.3]{Lurie:HTT}, whose objects are the presentable
  $\infty$-categories, and whose morphisms are the left adjoint functors, or
  equivalently colimit-preserving functors.  The $\infty$-category
  $\kat{Pr}^{\mathrm L}$ has an  `internal hom' (see \cite[5.5.3.8]{Lurie:HTT}): if
  $\CC$ and $\DD$ are presentable $\infty$-categories, $\Fun^{\mathrm
  L}(\CC,\DD)$, defined as the full subcategory of $\Fun(\CC,\DD)$ spanned by
  the colimit-preserving functors, is again presentable.  The mapping spaces in 
  $\kat{Pr}^{\mathrm L}$ are
  $\Map_{\kat{Pr}^{\mathrm L}}(\CC,\DD) = \Fun^{\mathrm L}(\CC,\DD)^\eq$.
  Finally, $\kat{Pr}^{\mathrm L}$ has a canonical symmetric monoidal structure,
  left adjoint to the closed structure.  See Lurie~\cite{Lurie:HA}, subsection
  4.8.1, and in particular 4.8.1.14 and 4.8.1.17.  The tensor product can be
  characterised as universal recipient of functors in two variables that
  preserve colimits in each variable, and we have an evaluation functor
  $$
  \CC \tensor \Fun^{\mathrm L}(\CC,\DD) \to \DD
  $$
  which exhibits $\Fun^{\mathrm L}(\CC,\DD)$ as an exponential of $\DD$ by 
  $\CC$.
  
  This tensor product has an easy description in the case of presheaf 
  categories (cf.~\cite[4.8.1.12]{Lurie:HA}): if $\CC= \PP(\CC_0)$ and $\DD= \PP(\DD_0)$ for small
  $\infty$-categories $\CC_0$ and $\DD_0$, then we have
  \begin{equation}\label{eq:tensor-times}
    \PP(\CC_0) \tensor \PP(\DD_0) \simeq \PP(\CC_0 \times \DD_0)  .
  \end{equation}
\end{blanko}

\begin{blanko}{The $\infty$-category $\LIN$.}\label{internalLin}
We define $\LIN$ to be the full
  subcategory of $\kat{Pr}^{\mathrm L}$ spanned by 
  the slices $\Grpd_{/S}$, for $S$ a locally finite $\infty$-groupoid.
  We call the functors {\em linear}.  The mapping spaces in $\LIN$ are
\begin{eqnarray*}
  \LIN(\Grpd_{/I}, \Grpd_{/J}) &=& \Fun^{\mathrm L}(\Grpd_{/I}, \Grpd_{/J})^\eq \\
  & \simeq & \Fun^{\mathrm L}(\Grpd^I, \Grpd^J)^\eq \\
  & \simeq & \Fun(I,\Grpd^J)^\eq \\
  & \simeq & (\Grpd^{I\times J})^\eq \\
  & \simeq & (\Grpd_{/I\times J} )^\eq.
\end{eqnarray*}
This shows in particular that the linear functors are given by spans.
Concretely, tracing through the chain of equivalences, a span defines a
left adjoint functor as described above in \ref{linearfunctorsfromspans}.
Composition in $\LIN$ is given by composing spans, i.e.~taking a
pullback.  This amounts to the Beck--Chevalley condition.

The $\infty$-category $\LIN$ inherits a symmetric monoidal closed structure 
from $\kat{Pr}^{\mathrm L}$.  For the `internal hom':
\begin{eqnarray*}
  \un\LIN(\Grpd_{/I}, \Grpd_{/J}) &:=& \Fun^{\mathrm L}(\Grpd_{/I}, \Grpd_{/J}) \\
  & \simeq & \Fun(I, \Grpd^{J}) \\
  &\simeq & \Fun( I \times J, \Grpd) \\
  &\simeq & \Grpd_{/I\times J} .
\end{eqnarray*}

Also the tensor product restricts, and we have the convenient formula
$$
\Grpd_{/I} \tensor \Grpd_{/J} = \Grpd_{/I\times J}
$$
with neutral object $\Grpd$.
This follows from formula~\eqref{eq:tensor-times} combined with the
fundamental equivalence $\Grpd_{/S} \simeq \Grpd^S$.

Clearly we have
$$
\LIN( \Grpd_{/I} \tensor \Grpd_{/J}, \Grpd_{/K}) \simeq \LIN( \Grpd_{/I}, 
\un\LIN(\Grpd_{/J},\Grpd_{/K})) 
$$
as both spaces are naturally equivalent to $(\Grpd_{/I\times J \times 
K})^\eq$.
\end{blanko}

\begin{blanko}{The linear dual.}\label{Lineardual}
`Homming' into the neutral object defines a contravariant autoequivalence
of $\LIN$:
\begin{eqnarray*}
  \LIN & \longrightarrow & \LIN\op  \\
  \Grpd_{/S} & \longmapsto & \un\LIN(\Grpd_{/S}, \Grpd) \simeq \Grpd_{/S}
  \simeq \Grpd^{S} .
\end{eqnarray*}
  Here there right-hand side should be considered the dual of $\Grpd_{/S}$.
  (Since our vector spaces are fully coordinatised, the difference between a
  vector space and its dual is easily blurred.  We will see a clearer difference
  when we come to the finiteness conditions, in which situation the dual of a `vector
  space' $\grpd_{/S}$ is $\grpd^S$ which should rather be thought of as a
  profinite-dimensional vector space.)
  
  For a span $S \stackrel p \leftarrow M \stackrel q \to T$ defining a
  linear functor $F := q\lowershriek \circ p\upperstar :\Grpd_{/S}\to
  \Grpd_{/T}$, the same span read backwards defines the {\em dual functor} $F^\vee :=
  p\lowershriek \circ q\upperstar : \Grpd^T \to \Grpd^S$.  Under the
  fundamental equivalence, this can also be considered a linear functor
  $F^t :\Grpd_{/T} \to \Grpd_{/S}$, called the {\em transpose} of
  $F$.
\end{blanko}

\begin{blanko}{Remark.}
  It is clear that there is actually an $(\infty,2)$-category in play here,
  with the $\un\LIN(\Grpd_{/S}, \Grpd_{/T})$ as hom $\infty$-categories.
  This can be described as a Rezk-category object in the `distributor'
  $\kat{Cat}$, following the work of Barwick and Lurie~\cite{Lurie:GoodwillieI}.
  Explicitly, let
  $\Lambda_k$ denote the full subcategory of $\Delta_k\times\Delta_k$
  consisting of the pairs $(i,j)$ with $i+j\leq k$.  These are the shapes of
  diagrams
  of $k$ composable spans.  They form a cosimplicial category.  
  Define $\operatorname{Sp}_k$ to be the full subcategory of
  $\Fun(\Lambda_k,\Grpd)$ consisting of those diagrams $S : \Lambda_k \to \Grpd$
  for which for all $i'< i$ and $j' < j$ (with  $i+j \leq k$)
  the square 
  $$\xymatrix{
     S_{i',j'} \drpullback \ar[r]\ar[d] & S_{i,j'} \ar[d] \\
     S_{i',j} \ar[r] & S_{i,j}
  }$$
  is a pullback.
  Then we claim that
  \begin{eqnarray*}
    \simplexcategory\op  & \longrightarrow & \kat{Cat}  \\
    {}[k] & \longmapsto & \operatorname{Sp}_k 
  \end{eqnarray*}
  defines a Rezk-category object in $\kat{Cat}$ corresponding to $\un\LIN$.
  We leave the claim unproved, as the result is not necessary for our purposes.
\end{blanko}

\section{Cardinality of finite $\infty$-groupoids}
\label{sec:finite}

\begin{blanko}{Finite $\infty$-groupoids.}\label{finite}
  An $\infty$-groupoid $X$ is called {\em locally finite} if at each base point
  $x$ the homotopy groups $\pi_i (X,x)$ are finite for $i\geq1$ and are trivial
  for $i$ sufficiently large.  An $\infty$-groupoid is called {\em finite} if it
  is locally finite and has finitely many components. 
  An example of a non locally finite $\infty$-groupoid is $B\Z$.

  Let $\grpd\subset \Grpd$ be the full subcategory spanned by the finite 
  $\infty$-groupoids.  For $S$ any $\infty$-groupoid, let $\grpd_{/S}$ be 
  the `comma $\infty$-category' defined by the
  following pullback diagram of $\infty$-categories:
    $$\xymatrix{
       \grpd_{/S} 
  \ar[r]\ar[d]\drpullback &       \Grpd_{/S} 
  \ar[d] \\
       \grpd \ar[r] & \Grpd .
    }$$
\end{blanko}

\begin{blanko}{Cardinality.} \cite{Baez-Dolan:finset-feynman}
  The {\em (homotopy) cardinality} of a finite $\infty$-groupoid $X$
is the nonnegative  rational number given by the formula
  $$
  \norm{X} := \sum_{x\in \pi_0 X} \prod_{i>0} \norm{\pi_i(X,x)}^{(-1)^i} .
  $$
  Here the norm signs on the right refer to order of homotopy groups.
\end{blanko}

If $X$ is a $1$-groupoid, that is, an $\infty$-groupoid
having trivial homotopy groups $\pi_i(X)=0$ for $i>1$, its cardinality is
  $$
  \norm{X} = \sum_{x\in \pi_0 X} \frac1{\norm{\Aut_X(x)}} .
  $$

  The notion and basic properties of homotopy cardinality have been around
for a long time. 
See in particular Baez--Dolan~\cite{Baez-Dolan:finset-feynman} and also 
To\"en~\cite{Toen:0501343}. 
The first printed reference we know of is Quinn~\cite[p.340]{Quinn:TQFT}.

\begin{lemma}\label{sumcard}
  A finite sum of finite $\infty$-groupoids is again finite, and
  cardinality is compatible with finite sums:
  $$
  \norm{\sum_{i=1}^n X_i}= \sum_{i=1}^n \norm{X_i} .
  $$
\end{lemma}
This is clear from the definition.

\begin{lemma}\label{FEB} Suppose $B$ is connected.
  Given a fibre sequence
  $$\xymatrix{
     F \ar[r]\ar[d]\drpullback & E \ar[d] \\
     1 \ar[r] & B ,
  }$$
  if two of the three spaces are finite then so is the third, and in that case
$$\norm E\;=\;\norm F\,\norm B.$$
\end{lemma}
\begin{proof}
  This follows from the homotopy long exact sequence of a fibre sequence.
\end{proof}

For $b\in B$, we denote by $B_{[b]}$ the connected component of $B$ containing
$b$.  Thus an $\infty$-groupoid $B$ is locally finite if and only if each connected
component $B_{[b]}$ is finite.
  
\begin{lemma}\label{lem:supp}
  Suppose $B$ locally finite.  Given a map $E\to B$, then $E$ is finite if and
  only if all fibres $E_b$ are finite, and are nonempty for only finitely many 
  $b\in\pi_0B$.  In this situation,
  $$\norm E\;=\;\sum_{b\in\pi_0(B)}\norm{E_b}\,\norm{B_{[b]}}.$$
\end{lemma}
\begin{proof}
  Write $E$ as the sum of the full fibres $E_{[b]}$,
  and apply 
  Lemma~\ref{FEB} to the fibrations
  $E_b\to E_{[b]}\to B_{[b]}$  for each $b\in\pi_0(B)$. Finally sum 
  (\ref{sumcard}) over those $b\in \pi_0 B$ with 
  non-empty $E_b$.
\end{proof}

\begin{cor}
  Cardinality preserves (finite) products.
\end{cor}

\begin{proof}
  Apply the Lemma~\ref{lem:supp} to a projection.
\end{proof}

\begin{blanko}{Notation.}
  Given any $\infty$-groupoid $B$ and a function $q: \pi_0 B \to 
\Q$, we write
$$
\int^{b\in B}q_b\;:=\; \sum_{b\in \pi_0 B} q_b\,\norm{B_{[b]}} 
$$ 
if the sum is finite. Then the previous lemma says
$$
\norm E\;=\;\int^{b\in B}\norm{E_b}
$$
for any finite $\infty$-groupoid $E$ and a map $E\to B$.
Two important special cases are given by fibre products and loop spaces:
\end{blanko}

\begin{lemma}\label{finite-products}
  In the situation of a pullback
  $$
\xymatrix{
X\times_B Y \ar[r] \drpullback \ar[d] & X\times Y  \ar[d] \\
B \ar[r]_-{\text{diag}} & B\times B,}$$
  if $X$ and $Y$ are finite, and $B$ is locally finite,
then $X\times_B Y$ is finite and 
$$\norm{X\times_B Y}\;=\;\int^{b\in B}\norm{X_b}\norm{Y_b}.$$
\end{lemma}

\begin{prop}\label{prop:closedunderfinlims}
  The $\infty$-category $\grpd$ of finite $\infty$-groupoids
  is closed under finite limits.
\end{prop}
\begin{proof}
  It is closed under pullbacks by the previous lemma, and it also contains the 
  terminal object, hence it is closed under all finite limits.
\end{proof}

\begin{lemma}\label{loop-finite}
  In the situation of a loop space
$$  \xymatrix{
\Omega(B,b) \ar[r] \drpullback \ar[d] & 1 \ar[d]^{\name b} \\
1 \ar[r]_-{\name b} & B_{[b]}\,.}$$ we have that
  $B$ is locally finite if and only if each $\Omega(B,b)$ is finite, and in that
  case
  $$
  \norm{\Omega(B,b)}\cdot\norm{B_{[b]}}=1.
  $$
\end{lemma}

\begin{blanko}{Finite maps.}
We say that a map $p: E \to B$ is  \emph{finite} if any pullback 
to a finite base $X$  has  finite total space $X'$, as in the diagram
\begin{align}\label{relfinpb}\vcenter{\xymatrix{
X' \ar[r] \drpullback \ar[d] & E  \ar[d]^p \\
X \ar[r]_c & B .}}\end{align}
The following two results are immediate.
\end{blanko}

\begin{lemma}
  If $B$ is finite and $E \to B$ is finite, then $E$ is finite.
  \hfill \qed
\end{lemma}
\begin{lemma}\label{lem:finmapbasechange}
  Finite maps are stable under base change.
  \hfil \qed
\end{lemma}

\begin{lemma}\label{lem:finitemaps}
A map $E\to B$ is finite if and only if each fibre $E_b$ is finite.
\end{lemma}
\begin{proof}
  The `only if' implication is a special case of \ref{lem:finmapbasechange}.
  If $p:E \to B$ has finite fibres, then also the map $X'\to X$ in the pullback
  diagram \eqref{relfinpb} has finite fibres $X'_x=E_{c(x)}$.  But since also
  $X$ is finite, Lemma~\ref{lem:supp} then implies that $X'$ is finite.
  Hence $p$ is finite.
\end{proof}

\begin{lemma}\label{locfinbase}
Suppose $p:E\to B$ has locally finite base. 
\begin{enumerate}
  \item If $p$ is finite then $E$ is locally finite. 
  \item If $E$ is finite then $p$ is finite.
\end{enumerate}
\end{lemma}
\begin{proof}
A full fibre $E_{[b]}$ of $p$ is finite if and only if $E_{b}$ is, by
Lemma~\ref{FEB}.  If each full fibre $E_{[b]}$ is finite, then each component
$E_{[e]}$ is, and if $E$ is finite then each full fibre is.
\end{proof}

\begin{lemma}\label{lem:1->B}
  $B$ is locally finite iff each name $1 \to B$ is a finite map.
\end{lemma}

\begin{prop}\label{prop:cartesianclosed}
  The $\infty$-category $\grpd$ of finite $\infty$-groupoids
  is cartesian closed.
\end{prop}
\begin{proof}
  We already know that $\Grpd$ is cartesian closed.
  We need to show that for $X$ and $Y$ finite $\infty$-groupoids, the mapping space
  $\Map(X,Y)$ is again finite.  We can assume  $X$ and $Y$ connected:
  indeed,
  if we write them as sums of their connected components,
  $X = \sum X_i$ and $Y= \sum Y_j$, then we have
  $$
  \Map(X,Y) = \Map (\sum X_i, Y) = \prod_i \Map(X_i,Y) 
  = \prod_i \sum_j \Map(X_i, Y_j)
  $$
  Since these are finite products and sums, if we can prove that each 
  $\Map(X_i,Y_j)$ is finite, then we are done.  
  Since $Y$ is finite, $\Map(S^k,Y)$ is finite for all $k\geq 0$, and
  there is $r \geq 0$ such that $\Map(S^k,Y)= *$
  for all $k\geq r$.  This is to say that $Y$ is $r$-truncated.
  On the other hand, since $X$ is finite, it has the homotopy type of
  a CW complex with finitely many cells in each dimension.  Write
  $$
  X = \colim_{i\in I} E_i
  $$
  for its realisation as a cell complex.
  Write $X' = \colim_{i\in I'} E_i$ for the colimit obtained by the same
  prescription 
  but omitting all cells of dimension $> r$; this is now a finite colimit,
  and the comparison map
  $X \to X'$ is $r$-connected.  Since $Y$ is $r$-truncated, we have
  $$
  \Map(X',Y) \isopil \Map(X,Y),
  $$
  and the first space is finite: indeed,
  $$
  \Map(X',Y) = \Map( \colim_{i\in I'} E_i, Y) = \lim_{i\in I'} \Map(E_i, Y)
  $$
  is a finite limit of finite spaces, hence is finite by 
  Proposition~\ref{prop:closedunderfinlims}.
 \end{proof}

\begin{theorem}
  For each locally finite $\infty$-groupoid $S$, the comma $\infty$-category
  $\grpd_{/S}$ of finite $\infty$-groupoids over $S$ is 
  cartesian closed.
\end{theorem}
\begin{proof}
  This is essentially a corollary of Proposition~\ref{prop:cartesianclosed}
  and the fact that the bigger $\infty$-category $\Grpd_{/S}$ is cartesian closed.
  We just need to check that the internal mapping object in $\Grpd_{/S}$
  actually belongs to $\grpd_{/S}$.   
  Given $a:A\to S$ and $b: B \to S$, the internal mapping object is
  $$
  \underline{\Map}_{/S}(a,b) \to S
  $$
  given fibrewise by
  $$
  \underline{\Map}_{/S}(a,b) _s = \Map(A_s, B_s)
  $$
  Since $A_s$ and $B_s$ are finite spaces, also the mapping space is finite,
  by \ref{prop:cartesianclosed}.
\end{proof}
\begin{cor}
  The $\infty$-category $\grpd$ is locally cartesian closed.
\end{cor}

\section{Finiteness conditions on $\infty$-groupoid slices}

\label{sec:finconslice}

In this section, 
after some motivation and background from linear algebra, we first explain the
finiteness conditions imposed on slice categories in order to model vector
spaces and profinite-dimensional vector spaces.  Then afterwards we assemble
all this into $\infty$-categories using more formal constructions.

\begin{blanko}{Duality in linear algebra.}\label{vect-rappels}
  There is a fundamental duality 
  $$
  \Vect \simeq \pro\vect\op
  $$
  between vector spaces and profinite-dimensional
vector spaces: given any vector space $V$, the linear dual $V\upperstar $ is a
profinite-dimensional vector space, and conversely, given a
profinite-dimensional vector space, its continuous dual is a vector space.
This equivalence is a formal consequence of the observation that the category 
$\vect$ of finite-dimensional vector spaces is self-dual: $\vect\simeq 
\vect\op$, and the fact that $\Vect= \ind\vect$, the ind completion of $\vect$.

In the fully coordinatised situation typical to algebraic combinatorics, the vector
space arises from a set $S$ (typically an infinite set of 
isoclasses of combinatorial objects): the vector space is then
$$V= \Q_S = 
\left\{\,\sum_{s\in S} c_s\,\delta_s\;:\;c_s\in\Q\text{ almost all zero}
\right\},$$
with basis the symbols $\delta_s$ for 
each $s\in S$.
The linear dual is the function
space $V\upperstar =\Q^S$, with canonical pro-basis consisting of
the functions $\delta^s$, taking the value 1 on $s$ and 0 elsewhere.

Vectors in $\Q_S$
are finite linear combinations of the $\delta_s$, and we represent a
vector as an infinite column vector $\vec v$ 
with only finitely many non-zero
entries.  A linear map $f:\Q_S \to \Q_T$ is given by 
matrix multiplication
$$
\vec v \mapsto A \cdot \vec v .
$$
for $A$ an infinite
$2$-dimensional matrix with $T$-indexed rows and $S$-indexed columns, and
with the crucial property that it is {\em column finite}: in each
column there are only finitely many non-zero entries. 
More generally, the matrix multiplication of two column-finite
matrices makes sense and is again a column-finite matrix.  The
identity matrix is clearly column finite.  A basis element
$\delta_s$ is identified with the
column vector all of whose entries are zero, except the one of index
$s$.

On the other
hand, elements in the function space $\Q^S$ are represented as
infinite row vectors. 
The continuous linear map $\Q^T \to \Q^S$,
dual to the linear map $f$, is represented by the {\em same matrix} 
$A$, but viewed now as sending a row vector $\vec w$ (indexed by $T$)
to the matrix product $\vec w \cdot A$.  Again the fact that $A$ is 
column finite ensures that this matrix product is well defined.

There is a canonical perfect pairing
\begin{eqnarray*}
    \Q_S \times \Q^S  &\longrightarrow& \Q \\
    (\vec v , f) &\longmapsto& f(\vec v)
\end{eqnarray*}
given by evaluation. In matrix terms, it is just a matter of 
multiplying $f\cdot \vec v$.
\end{blanko}

This duality has a very neat description in homotopy linear algebra
over $\grpd$, the $\infty$-category of finite $\infty$-groupoids.
While the vector space $\Q_{\pi_0 S}$ is modelled by the
$\infty$-category $\grpd_{/S}$, the function space $\Q^{\pi_0 S}$ is
modelled by the $\infty$-category $\grpd^S$.  

The classical duality results from taking cardinality of a duality on the
categorical level, that we proceed to explain.  For the most elegant definition
of cardinality we first need to introduce the objective versions of $\Vect$ and
$\pro\vect$.  These will be $\infty$-categories $\ind\lin$ whose objects are of
the form $\grpd_{/S}$, and $\pro\lin$ whose objects are of the form $\grpd^S$.

\bigskip

We shall need also the following variations.
For $S$ a locally finite $\infty$-groupoid,
we are concerned with the following $\infty$-categories.

\begin{itemize}
  \item $\grpd_{/S}$: the slice $\infty$-category of morphisms
  $\sigma \to S$, with $\sigma$ finite.
  \item $\grpd^S$: the full subcategory of $\Grpd^S$ spanned by the presheaves
  $S\to\Grpd$ whose images lie in $\grpd$.
  \item $\Grpd^{\relfin}_{/S}$: the full subcategory of $\Grpd_{/S}$ 
  spanned by the finite maps $p:X\to S$.
  \item $\grpd^S_{\finsup}$: the full subcategory of $\Grpd^S$ 
  spanned by presheaves with finite values and finite support.
  By the support of a presheaf $F: S \to \Grpd$ we mean the
  full subgroupoid of $S$ spanned by the objects $x$ for which
  $F(x)\neq\emptyset$.
\end{itemize}

\begin{prop}\label{fundamentalfinite}
  The fundamental equivalence $\Grpd_{/S} \simeq \Grpd^S$
  restricts to equivalences
  $$
  \Grpd^{\relfin}_{/S} \simeq \grpd^S  \qquad \text{ and } \qquad
  \grpd_{/S} \simeq \grpd^S_{\finsup}
  $$
\end{prop}
\begin{proof}
  The inclusions $\Grpd^{\relfin}_{/S} \subset \Grpd_{/S}$ 
  and $\grpd^S \subset \Grpd^S$
  are both full, and the objects characterising them correspond to each other
  under the fundamental equivalence because of 
  Lemma~\ref{lem:finitemaps}.
  Similarly, the inclusions $\grpd_{/S}\subset \Grpd_{/S}$ and 
  $\grpd^S_{\finsup} \subset \Grpd^S$ are both full, 
  and the objects characterising them correspond to each other
  under the fundamental equivalence, this time in virtue of 
  Lemma~\ref{lem:supp}.
\end{proof}

\begin{prop}\label{finitetypespan}
    For a span $S \stackrel p\leftarrow M \stackrel q\to T$ of locally finite 
    $\infty$-groupoids,
the following are equivalent:
\begin{enumerate}
\item $p$ is finite
\item The linear functor $F:= q\lowershriek \circ p\upperstar :
\Grpd_{/S}\to \Grpd_{/T}$ restricts to
$$
\grpd_{/S}\stackrel{p\upperstar}\longrightarrow 
\grpd_{/M}\stackrel{q\lowershriek}\longrightarrow \grpd_{/T}
$$
\item The transpose $F^t := p\lowershriek\circ q\upperstar
: \Grpd_{/T} \to \Grpd_{/S}$
restricts to
$$
\Grpd^{\relfin}_{/T}
\stackrel{q\upperstar}\longrightarrow 
\Grpd^{\relfin}_{/M}\stackrel{p\lowershriek}\longrightarrow 
\Grpd^{\relfin}_{/S}
$$
\item The dual functor $F^\vee: \Grpd^T \to \Grpd^S$ restricts to 
$$
\grpd^T
\to
\grpd^S
$$
\item The dual of the transpose, $F^t{}^\vee :\Grpd^S \to \Grpd^T$ restricts to
$$
\grpd^S_{\finsup} 
\to
\grpd^T_{\finsup}
$$
\end{enumerate}
\end{prop}
\begin{proof}
  The biimplications (1)$\Leftrightarrow$(2) and (1)$\Leftrightarrow$(3)
  follow from the definition of finite map.
  The biimplications  (2)$\Leftrightarrow$(5) and (3)$\Leftrightarrow$(4)
  follow from the equivalences in Proposition~\ref{fundamentalfinite}.
\end{proof}

\begin{blanko}{Finite homotopy sums.}
  The $\infty$-category $\grpd_{/S}$ has finite homotopy sums: for $I$ finite
  and $F:I \to \grpd_{/S}$ we have $\colim F = p\lowershriek (X \to I \times
  S)$, where $p:I\times S \to S$ is the projection.  A family $X \to I \times S$
  comes from some $F:I \to \grpd_{/S}$ and admits a homotopy sum in $\grpd_{/S}$
  when for each $i\in I$, the partial fibre $X_i$ is finite.  Since already $I$
  was assumed finite, this is equivalent to having $X$ finite.

  The following is the finite version of Proposition~\ref{HoSum}
\end{blanko}

\begin{lemma}\label{hosum}
  The $\infty$-category  $\grpd_{/S}$ is the finite-homotopy-sum completion of $S$.
\end{lemma}

\section{$\infty$-categories of linear maps with infinite $\infty$-groupoid coefficients}

\label{sec:weird}

Our main interest is in the linear $\infty$-categories with finite
$\infty$-groupoid coefficients, but it is technically simpler to introduce first
the infinite-coefficients version of these $\infty$-categories, since they can
be defined as subcategories in $\LIN$, and can be handled with the ease of
presentable $\infty$-categories.

\bigskip

Recall that a span $(S\stackrel p\longleftarrow M \stackrel q\longrightarrow T)$
defines a linear functor
$$
\Grpd_{/S}\stackrel{p\upperstar}\longrightarrow 
\Grpd_{/M}\stackrel{q\lowershriek}\longrightarrow \Grpd_{/T}.
$$

Let $\Lin \subset \LIN$ be the $\infty$-category whose objects are the slices
$\Grpd_{/\sigma}$, with $\sigma$ finite, and whose morphisms are 
those linear functors between them which preserve finite objects.
Clearly these are given by
the spans
of the form $\sigma \leftarrow \mu \rightarrow \tau$ where $\sigma, \tau$ and
$\mu$ are finite.  Note that there are equivalences of $\infty$-categories
$\Grpd_{/\sigma}\simeq \Grpd^{\sigma}$ for each finite $\sigma$.

From now on we adopt the blanket convention that Greek letters denote finite 
$\infty$-groupoids.

\bigskip

Let $\ind\Lin$ be the $\infty$-category whose objects are the slices $\Grpd_{/S}$ with 
$S$
locally finite, and whose
morphisms are the linear functors between them that preserve finite objects.
These correspond to the spans of the form $S\stackrel p\leftarrow M \rightarrow 
T$ with $p$ finite.   

Let $\pro\Lin$ be the $\infty$-category whose objects are the presheaf categories
$\Grpd^{S}$ with $S$ locally finite, and whose morphisms are the continuous 
linear functors:

A linear functor $F: \Grpd^{T}\to \Grpd^{S}$ is called \emph{continuous} when for all 
$\varepsilon \subset S$ there exists 
$\delta \subset T$ and a factorisation 
$$\xymatrix{
      \Grpd^{T} \ar[r]\ar[d]_F &  \Grpd^{\delta} \ar[d]^{F_{\delta}} \\
      \Grpd^{S}\ar[r] &  \Grpd^{\varepsilon}
  .}$$
Here we quantify over finite groupoids $\varepsilon$ and $\delta$ with full inclusions into $S$ and $T$; the horizontal maps are the projections of the canonical pro-structures.

\begin{prop}\label{prop:cont}
    For a linear functor $F: \Grpd^T \to \Grpd^S$ in $\LIN$, represented by a 
    span
    $$
    S \stackrel p \leftarrow M \stackrel q \to T,
    $$
    the following are equivalent.
    
    \begin{enumerate}
    
        \item The span is of finite type (i.e.~$p$ is a finite map).
    
        \item $F$ is continuous.
    \end{enumerate}
\end{prop}
\begin{proof}
    It is easy to see that if the span is of finite type then  $F$ is continuous: 
    for any given
    finite $\varepsilon \subset S$ with inclusion $j$, the pullback $\mu$
    is finite, and we can take $\delta$ to be the essential full image 
    of the composite $q \circ m$: 
     \begin{equation}\label{epsdeltdiag}
	 \xymatrix{
    \varepsilon \ar[d]_j & \mu \dlpullback \ar[l]_{\bar p} 
    \ar[d]^m  \ar[r]^{\bar q} & \delta\ar[d]^i\\
    S & \ar[l]^p M \ar[r]_q & T .}
    \end{equation}
    Now by Beck-Chevalley,
    $$j\upperstar 
    p\lowershriek q\upperstar = \bar p\lowershriek m\upperstar q \upperstar =
     \bar p \lowershriek \bar q{}\upperstar i \upperstar$$ 
    which is precisely the continuity condition.

    Conversely, if the factorisation in the continuity diagram exists,
    let $\varepsilon \leftarrow \mu \to \delta$ be the span (of finite 
    $\infty$-groupoids) representing $f_{\delta_\varepsilon}$.  Then we have the
    outer rectangle of the diagram~\eqref{epsdeltdiag} and an 
    isomorphism
    $$j\upperstar 
    p\lowershriek q\upperstar = 
     \bar p \lowershriek \bar q{}\upperstar i \upperstar$$ 
    Now a standard argument implies the existence of $m$ completing
    the diagram: namely take the pullback of $j$ and $p$, with the 
    effect of interchanging the order of upperstar and lowershriek.
    Now both
    linear maps are of the form upperstars-followed-by-lowershriek,
    and by uniqueness of this representation, the said pullback must 
    agree with $\mu$ and in particular is finite.  Since this is true
    for every $\varepsilon$, this is precisely to say that $p$ is 
    a finite map.
\end{proof}

The continuity condition is precisely continuity for the profinite 
topology, as we proceed to explain.
Every locally finite $\infty$-groupoid $S$ is canonically the filtered colimit of
its finite full subgroupoids:
$$
S = \colim_{\alpha\subset S} \alpha .
$$
Similarly, $\Grpd^S$ is a cofiltered limit of $\infty$-categories $\Grpd^\alpha$:
$$
\Grpd^S = \lim_{\alpha\subset S} \Grpd^\alpha .
$$
This leads to the following `categorical' description of the mapping spaces 
(compare~SGA4~\cite{SGA4.1}, Exp.1):
$$
\pro\Lin (\Grpd^T, \Grpd^S) : = \lim_{\varepsilon\subset S} \colim_{\delta\subset 
T} \Lin(\Grpd^\delta, \Grpd^\varepsilon) .
$$

\section{$\infty$-categories of linear maps with finite $\infty$-groupoid coefficients}
\label{sec:lin}

In this section we shall work with coefficients
in $\grpd$, the $\infty$-category of finite $\infty$-groupoids.

\begin{blanko}{The $\infty$-category $\lin$.}
  Let $\bigcat$ denote the (very large) $\infty$-category of possibly 
  large $\infty$-categories.
  We define $\lin$ to be the subcategory of $\bigcat$ whose objects are those
$\infty$-categories equivalent to $\grpd_{/\sigma}$ for some finite $\infty$-groupoid $\sigma$,
and whose mapping spaces are the full subgroupoids of those of $\bigcat$ given by
the functors which are restrictions of functors in $\Lin(\Grpd_{/\sigma}, 
\Grpd_{/\tau})$.  Note that the latter mapping space was exactly defined as those
linear functors in $\LIN$ that preserved finite objects.  Hence, by 
construction there is an equivalence of mapping spaces 
$$
\lin(\grpd_{/\sigma},\grpd_{/\tau}) \simeq 
\Lin(\Grpd_{/\sigma}, 
\Grpd_{/\tau}) ,
$$
and in particular, the mapping spaces are given by spans of finite $\infty$-groupoids.
The maps can also be described as those functors that 
preserve finite homotopy sums.  By construction we have an equivalence
of $\infty$-categories
$$
\lin \simeq \Lin .
$$
\end{blanko}

\begin{blanko}{The $\infty$-category $\ind\lin$.}
  Analogously, we define $\ind\lin$ to be the subcategory of $\bigcat$, whose
  objects are the $\infty$-categories equivalent to $\grpd_{/S}$ for some locally finite
  $\infty$-groupoid $S$, and whose mapping spaces are the full subgroupoids of the
  mapping spaces of $\bigcat$ given by the functors that are restrictions of
  functors in $ \Lin(\Grpd_{/S}, \Grpd_{/T})$; in other words 
  (by~\ref{finitetypespan}), they are the
  $\infty$-groupoids of spans of finite type.  Again by construction we have
  $$
  \ind\lin \simeq \ind\Lin .
  $$
\end{blanko}

\begin{blanko}{$\infty$-categories of prolinear maps.}
  We denote by $\pro\lin$ the $\infty$-category whose objects are the $\infty$-categories
  $\grpd^{S}$, where $S$ is locally finite, and whose morphisms are restrictions
  of continuous linear functors.
  We have seen that the mapping spaces are given by spans of finite type:
  $$\pro\lin(\grpd^T  ,\grpd^{S}  )\;\;=\;\;\left\{  
  (T\stackrel q\longleftarrow M \stackrel p\longrightarrow S
  )\;:\; p\text{ finite}\right\}.$$

  As in the ind case we have
  $$
  \pro\lin \simeq \pro\Lin ,
  $$
  and by combining the previous results we also find
  $$
  \pro\lin (\grpd^T, \grpd^S) : = \lim_{\varepsilon\subset S} \colim_{\delta\subset 
  T} \lin(\grpd^\delta, \grpd^\varepsilon) .
  $$
\end{blanko}

\begin{blanko}{Mapping $\infty$-categories.}
  Just as $\bigcat$ has internal mapping $\infty$-categories, 
whose maximal subgroupoids are the  mapping spaces,
  we also have internal mapping $\infty$-categories in $\lin$, denoted
  $\un{\lin}$:
  $$
  \un\lin (\grpd_{/\sigma},\grpd_{/\tau}) \simeq \grpd_{/\sigma\times\tau}.
  $$

  Also $\ind\lin$ and $\pro\lin$ have mapping $\infty$-categories, but due to the
  finiteness conditions, they are not internal.  The mapping $\infty$-categories (and mapping spaces) are
  given in each case as $\infty$-categories (respectively $\infty$-groupoids) of spans of finite type.
Denoting the
  mapping categories with underline, we content ourselves to record the 
  important case of `linear dual':
\end{blanko}

\begin{prop}\label{linearduals}
  \begin{eqnarray*}
    \un{\ind\lin}(\grpd_{/S},\grpd) &=& \grpd^S\\
    \un{\pro\lin}(\grpd^T,\grpd) &=& \grpd_{/T}.
  \end{eqnarray*}
\end{prop}

\begin{blanko}{Remark.}
  It is clear that the correct viewpoint here would be that there is 
  altogether a $2$-equivalence between the 
  $(\infty,2)$-categories
  $$
  \ind\lin\op\cong\pro\lin
  $$ 
  given on objects by $\grpd_{/S}\mapsto\grpd^{S}$, 
  and by the identity on homs.  It all comes formally from an ind-pro like duality
  starting with the anti-equivalence
  $$
  \lin \simeq \lin\op.
  $$
(Since we only (co)complete over filtered diagrams of monomorphisms, this is not precisely ind-pro duality.)

  Taking $S=1$ we see that $\grpd$ is an object of both $\infty$-categories, and
  mapping into it gives the duality isomorphisms of 
  Proposition~\ref{linearduals}.
\end{blanko}

\begin{blanko}{Monoidal structures.}
  The $\infty$-category $\ind\lin$ has two monoidal structures: $\oplus$ and
  $\tensor$, where $\grpd_{/S} \oplus \grpd_{/T} = \grpd_{/S+T}$ and $\grpd_{/S}
  \tensor \grpd_{/T} = \grpd_{/S\times T}$.  The neutral object for the first is
  clearly $\grpd_{/0} = 1$ and the neutral object for the second is
  $\grpd_{/1}=\grpd$.  The tensor product distributes over the direct sum.  The
  direct sum is both the categorical sum and the categorical product (i.e.~is a
  biproduct).  There is also the operation of infinite direct sum: it is the
  infinite categorical sum but not the infinite categorical product.  This is analogous to vector spaces.

  Similarly, also the $\infty$-category $\pro\lin$ has two monoidal structures, $\oplus$
  and $\tensor$, given as $\grpd^S \oplus \grpd^T = \grpd^{S+T}$ and $\grpd^S
  \tensor \grpd^T = \grpd^{S\times T}$.  The $\tensor$ should be considered the
  analogue of a completed tensor product.  Again $\oplus$ is both the
  categorical sum and the categorical product, and $\tensor$ distributes over
  $\oplus$.  Again the structures allow infinite versions, but this times the
  infinite direct sum is a categorical infinite product but is not an infinite
  categorical sum.

  To see the difference between the role of infinite $\oplus_i$ in
  $\ind\lin$ and in $\pro\lin$, consider the following. In $\pro\lin$ there is a diagonal map
  $\grpd^S \to \oplus_i \grpd^S = \grpd^{\sum_i S}$ given by sending a
  presheaf $s \mapsto X_s$ to the presheaf on $\sum_i S_i$ given by
  $(i,s) \mapsto X_s$.  Under the equivalence $\grpd^S \simeq
  \Grpd_{/S}^{\relfin}$, this corresponds to the assignment sending a
  finite map $X \to S$ to $\sum_i X \to \sum_i S$, which is of coruse
  again a finite map.  But this does not make sense in $\ind\lin$
  since $\sum_i X$ is not generally finite.  On the other hand,
  $\ind\lin$ sports a codiagonal $\oplus_i \grpd_{/S} = \grpd_{/\sum_i
  S} \to \grpd_{/S}$ given by sending $\alpha \to \sum_i S$ to the
  composite $\alpha \to \sum_i S \to S$, where the second map is the
  codiagonal for the infinite sum of $\infty$-groupoids.  Composing
  with $\sum_i S \to S$ obviously does not alter the finiteness of
  $\alpha$, so there is no problem with this construction.  In
  contrast, this construction does not work in $\pro\lin$: for a
  presheaf $X \in \grpd^{\sum_i S}$ given by sending $(i,s)$ to a
  finite $\infty$-groupoid $X_{i,s}$, the assignment $s
  \mapsto \sum_i X_{i,s}$ will generally not take finite values.

\end{blanko}

\begin{blanko}{Summability.}
  In algebraic combinatorics, profinite notions are often expressed in terms of
  notions of summability.  We briefly digress to examine our constructions from this
  point of view.
 
  For $B$ a locally finite $\infty$-groupoid, a $B$-indexed family $g : E \to B \times I$
  (as in \ref{scalar&hosum}) is called {\em summable} if the composite $E \to B
  \times I \to I$ is a finite map.  The condition implies that in fact the
  members of the family were already finite maps.  Indeed, with reference to the
  diagram
    $$\xymatrix{
    E_{b,i} \drpullback \ar[r] \ar[d] & E_i \drpullback \ar[r] \ar[d] 
    & E \ar[d] \\
    \{b\}\times \{i\} \ar[r] & B \times \{i\} \ar[r] \ar[d] 
    \drpullback & B\times I \ar[d] \\
    & \{i\} \ar[r] & I
    }$$
    summability 
    implies (by Lemma~\ref{lem:finitemaps})
    that each $E_i$ is finite,
    and therefore (by Lemma~\ref{lem:1->B} since $B$ is locally finite) we 
    also conclude that each $E_{b,i}$ is finite, which is precisely to
    say that the members $g_b :E_b \to I$ are finite  maps
    (cf.~\ref{lem:finitemaps} 
    again).
    It thus makes sense to interpret the family as a family of 
    objects in $\Grpd_{/I}^{\relfin}$.  And finally we can say that 
    a summable family is a family $g:E\to B \times I$
    of finite maps $g_b : E_b \to I$, whose 
    homotopy sum
    $p\lowershriek(g)$ is again a finite map.
If $I$ is finite, then the only summable families are the finite 
families (i.e.~$E \to B \times I$ with $E$ finite).
A family $g : E \to B\times I$, given equivalently as a functor
$$
F: B \to \grpd^I,
$$
is summable if and only if it is a cofiltered limit of diagrams $F_\alpha: B \to
\grpd^{\alpha}$ (with finite $\alpha$ and full $\alpha\subset I$).

It is easy to check that
    a map $q:M \to T$ (between locally finite $\infty$-groupoids) is 
    finite if and only if for every finite map 
    $f:X \to M$ we have that also $q\lowershriek(f)$ is 
    finite.
Hence we find
\end{blanko}
\begin{lemma}
A span $I \stackrel p \leftarrow M \stackrel q \to J$
preserves summable families if and only if $q$ is finite.
\end{lemma}

\section{Duality}
\label{sec:duality}

Recall that $\grpd$ denotes the $\infty$-category of
finite $\infty$-groupoids.

\begin{blanko}{The perfect pairing.}\label{duality}
  We have a perfect pairing
\begin{eqnarray*}
    \grpd_{/S} \times \grpd^S  &\longrightarrow& \grpd \\
    (p , f) &\longmapsto& f(p)
\end{eqnarray*}
given by evaluation.
In terms of spans, write the map-with-finite-total-space $p:\alpha\to S$
as a finite span $1 \leftarrow \alpha \stackrel p \to S$, and write the presheaf 
$f:S\to \grpd$ as the finite span $S \stackrel f \leftarrow F \to 1$, where $F$ 
is
the total space of the Grothendieck construction of $f$.
(In other words, the functor $f$ on $S$ corresponds to a linear functor
on $\grpd_{/S}$, so write it as the representing span.)
Then the evaluation is given by composing these two spans, and hence
amounts just to taking the pullback of $p$ and $f$.

The statements mean: for each $p:\alpha \to S$ in $\grpd_{/S}$, the map
\begin{eqnarray*}
  \grpd^S & \longrightarrow & \grpd  \\
  f & \longmapsto & f(p)
\end{eqnarray*}
is prolinear, and
the resulting functor
\begin{eqnarray*}
  \grpd_{/S} & \longrightarrow & \un{\pro\Lin}(\grpd^S,\grpd)  \\
  p & \longmapsto & \qquad\!\! (f \mapsto f(p))
\end{eqnarray*}
is an equivalence of $\infty$-categories (by~Proposition~\ref{linearduals}).

Conversely, for each $f: S \to \grpd$ in $\grpd^S$, the map
\begin{eqnarray*}
  \grpd_{/S} & \longrightarrow & \grpd  \\
  p & \longmapsto & f(p)
\end{eqnarray*}
is linear, and the resulting functor
\begin{eqnarray*}
  \grpd^S & \longrightarrow & \un{\ind\Lin}(\grpd_{/S},\grpd)  \\
  f & \longmapsto &  \qquad\!\! (p \mapsto f(p))
\end{eqnarray*}
is an equivalence of $\infty$-categories (by~Proposition~\ref{linearduals}).
\end{blanko}

\begin{blanko}{Remark.}\label{dualduality}
  By the equivalences of Proposition~\ref{fundamentalfinite}, we also get the
perfect pairing
\begin{eqnarray*}
    \Grpd_{/S}^{\relfin} \times \grpd^S_{\finsup}  &\longrightarrow& \grpd \\
    (p , f) &\longmapsto& f(p) .
\end{eqnarray*}  
\end{blanko}

\begin{blanko}{Bases.}
  Both $\grpd_{/S}$ and $\grpd^S$ feature a canonical basis, actually an
essentially unique basis.  The basis elements in $\grpd_{/S}$ are the names
$\name s : 1 \to S$: every object $p:X \to S$ in $\grpd_{/S}$ can be written as a finite 
homotopy linear combination
$$
p = \int^{s\in S} \norm{X_s} \ \name s .
$$
Similarly, in $\grpd^S$, the representables $h^t:= \Map(t, - )$ form a basis:
every presheaf on $S$ is a colimit, and in fact a homotopy sum, of such
representables.
These bases are dual to each other, except for a normalisation: if $p=\name s$
and $f=h^t = \Map(t, - )$, then they pair to
$$
\Map(t,s) \simeq
\begin{cases}
    \Omega(S,s) & \text{ if } t\simeq s \\
    0 & \text{ else. }
\end{cases}
$$
The fact that we obtain the loop space $\Omega(S,s)$ instead of $1$ is actually a feature:
we shall see below that on taking cardinality we obtain
the canonical pairing
\begin{eqnarray*}
  \Q_S \times \Q^S & \longrightarrow & \Q \\
  (\delta_i,\delta^j) & \longmapsto & \begin{cases} 1 & \text{if } i=j \\
  0 & \text{else.}\end{cases}
\end{eqnarray*}
\end{blanko}

\section{Cardinality as a functor}
\label{sec:metacard}

Recall that $\grpd$ denotes the $\infty$-category of
finite $\infty$-groupoids.
The goal now is that each slice $\infty$-category $\grpd_{/S}$, 
and each finite-presheaf $\infty$-category $\grpd^S$, 
should have a notion of homotopy cardinality 
with values in the vector space $\Q_{\pi_0 S}$, 
and in the profinite-dimensional vector space $\Q^{\pi_0 S}$, respectively.  
The idea of Baez, Hoffnung and Walker~\cite{Baez-Hoffnung-Walker:0908.4305} 
is to achieve this by a `global' assignment, 
which in our setting this amounts to functors
$\ind\lin \to \Vect$ and $\pro\lin\to \pro\vect$.  
By the observation that families are special cases of spans, 
just as vectors can be identified with linear maps from the ground field,
this then specialises to define a `relative' cardinality 
on every slice $\infty$-category.

\begin{blanko}{Definition of cardinality.}\label{metacard}
  We define {\em meta cardinality}
  $$
  \|\quad\|:\ind\lin\to\Vect
  $$
  on objects by
  $$
  \|\grpd_{/T}\|:=\Q_{\pi_0 T},
  $$
  and on morphisms by taking a finite-type span 
  $S \stackrel p \leftarrow M\stackrel q\to T$
  to the linear map
  \begin{eqnarray*}
    \Q_{\pi_0 S} & \longrightarrow & \Q_{\pi_0 T}  \\
    \delta_s & \longmapsto & \int^t \norm{M_{s,t}} \delta_t = \sum_t 
    \norm{T_{[t]}} \norm{M_{s,t}} \delta_t ,
  \end{eqnarray*}
  with associated matrix $A_{t,s} :=
  \norm{T_{[t]}} \norm{M_{s,t}}$.
    
  Dually,
  $$
  \|\quad\|:\pro\lin\to\pro\vect
  $$
  is defined on objects by
  $$
  \|\grpd^S\|:=\Q^{\pi_0S},
  $$
  and on morphisms by the assigning the {\em same} matrix to a finite-type span as 
  before.
\end{blanko}

\begin{prop}\label{cardfunctor}
  The meta cardinality assignments just defined
  $$
  \|\quad\|:\ind\lin\to\Vect,\qquad\qquad \|\quad\|:\pro\lin\to\pro\vect
  $$
  are functorial.
\end{prop}
\begin{proof}
  First observe that the functor is well defined on morphisms.  Given a
  finite-type span $S \stackrel p \leftarrow M\stackrel q\to T$ defining the linear
  functors $L: \grpd_{/S} \to \grpd_{/T}$ and $L^\vee : \grpd^T \to
  \grpd^S$, the linear maps
  $$
  \|L\|:\Q_{\pi_0S}\longrightarrow\Q_{\pi_0T}
  ,\qquad\qquad 
  \|L^\vee \|:\Q^{\pi_0T}\longrightarrow\Q^{\pi_0S}
  $$
  are represented 
by the same matrix $\|L\|_{t,s}=
		      \norm{M_{s,t}}      \norm{T_{[t]}} 
  $ with respect to the given (pro-)bases,
  $$
  \|L\|\,\biggl(\sum_{s\in\pi_0S} c_s\,\delta_s\biggr)
  \;=\;
  \sum_{s,t} c_s\,
		       \norm{M_{s,t}}      \norm{T_{[t]}} 
  \,\delta_t 
  \,,
  $$
  and
  $$
  \|L^\vee\| 
  \,\biggl(\sum_{t\in\pi_0T} c_t\,\delta^t\biggr)
  \;=\;
  \sum_{s,t} c_t\,
		       \norm{M_{s,t}}      \norm{T_{[t]}} 
  \,\delta^s \,.
  $$
  These sums make sense as the matrix $\norm{M_{s,t}}\norm{T_{[t]}}$
  has finite entries and is column-finite: for each $s\in\pi_0S$ the fibre $M_s$
  is finite so the map $M_{s}\to T$ is finite by Lemma \ref{locfinbase}, and the
  fibres $M_{s,t}$ are non-empty for only finitely many $t\in\pi_0T$.  

  Now $\Vect$ and $\pro\vect$ are $1$-categories, so we observe that  
$\ind\lin \to \Vect$ and $\pro\lin\to \pro\vect$ are well defined since they are well defined on the homotopy categories (equivalent spans define the same matrix). It remains to check functoriality:
 The identity span
  $L=(S\leftarrow S\to S)$ gives the identity matrix: $\|L\|_{s_1,s_2}=0$ if
  $s_1,s_2$ are in different components, and
  $\|L\|_{s,s}=\norm{\Omega(S,s)}\norm{S_{[s]}}=1$ by Lemma~\ref{loop-finite}.
  Finally, composition of spans corresponds to matrix product: for $L=(S\leftarrow M\rightarrow T)$ and
  $L'=(T\leftarrow N\rightarrow U)$ we have
  $$
  \norm{(M\times_TN)_{s,u}}=\int^{t\in T}\norm{M_{s,t}\times N_{t,u}}
  =
  \sum_{t\in\pi_0T}\norm{M_{s,t}} \norm{T_{[t]}}\norm{N_{t,u}}
  $$
  and so $\displaystyle
  \|L'L\|_{u,s}=\sum_{t\in\pi_0T}
  \norm{M_{s,t}}\norm{T_{[t]}}\norm{N_{t,u}}\norm{U_{[u]}}
  =\sum_{t\in\pi_0T}\|L'\|_{u,t}\|L\|_{t,s}
  $.
\end{proof}

\begin{blanko}{Cardinality of families.}
As a consequence of  Proposition~\ref{cardfunctor} 
we obtain, given any locally finite $\infty$-groupoid $T$, a notion of cardinality of any $T$-indexed family,
$$
\norm{ \ \ } : \grpd_{/T}  \longrightarrow \normnorm{\grpd_{/T}}= \Q_{\pi_0 T}.
$$
To define this function we observe that an object  $x:X\to T$ in $\grpd_{/T}$ can be identified with a finite-type span  $L_x$ of the form $1\leftarrow X\xrightarrow{\,x\,}T$, and conversely its meta cardinality $\|L_x\|$ is a linear map 
$\Q_{\pi_01}\to\Q_{\pi_0T}$, which can be identified with a vector in  $\Q_{\pi_0T}$.
That is, we set
$$
\norm{x}\;:=\;\|L_x\|\,\left(\delta_1\right).
$$
By the definition of $\|L\|$ in Proposition \ref{cardfunctor}, we can write
$$
\norm x
\;\;=\;\;
\sum_{t\in \pi_0T}
|X_t|\;|T_{[t]}|\;\delta_t\;\;=\;\;\int^{t\in T}\!|X_t|\;\delta_t
$$
\end{blanko}

\begin{lemma} Let $T$ be a locally finite $\infty$-groupoid.
\begin{enumerate}
\item
  If $T$ is connected, with $t\in T$, and $x:X \to T$ in $\grpd_{/T}$, then 
  $$
  \norm x \;=\; \norm X \, \delta_t \;\; \in\;\;\Q_{\pi_0 T}.
  $$
\item  The cardinality of $\name t:1\to T$ in $\grpd_{/T}$
  is the basis vector $\delta_t$.
\end{enumerate}
\end{lemma}
\begin{proof}
(1)  By definition, $\norm x\;=\;\norm{X_t}\,\norm T\, \delta_t$, and by 
Lemma~\ref{FEB}, this is $\norm X \,\delta_t$ 
\par\noindent
(2)  The fibre of $\name t$ over $t'$ is empty except when $t,t'$ are in the same component, so we reduce to the case of connected $T$ and apply (1).
\end{proof}
Since meta cardinality is functorial, we obtain the following property of local cardinality.
\begin{lemma}
Let $S,T$ be locally finite $\infty$-groupoids, and $L:\grpd_{/S}\to \grpd_{/T}$ a linear functor. Then, for any $x:X\to S$ in  $\grpd_{/S}$ we have
$$
\norm{L(x)}
\;\;=\;\;\|L\|(\norm{x}).
$$
\end{lemma}
\begin{proof}
The family $y=L(x)$ in $\grpd_{/T}$ corresponds to a span $L_y$ of the form $1\leftarrow Y\to T$, given by the composite of the span $L_x$ and that defining $L$. Hence, by functoriality 
$\|L_y\|(\delta_1)=\|L\|\,\|L_x\|(\delta_1)$, as required.
\end{proof}

\begin{blanko}{Cardinality of presheaves.}
  We also obtain a notion of cardinality of presheaves: for each $S$, define
    $$
  \norm{ \ \ } : \grpd^S  \longrightarrow \normnorm{\grpd^S}= \Q^{\pi_0 S},
  \qquad \norm{f}\;:=\;\|L_f\|.
  $$
Here $f: S \to \grpd$ is a presheaf, and
$L_f : \grpd_{/S} \to \grpd$ its extension by linearity;
$L_f$ is given by the span $S \leftarrow 
F \to 1$, where $F\to S$ is the Grothendieck construction of $f$.
The meta cardinality of this span is then a linear map $\Q_{\pi_0 S} 
\to \Q_1$, or equivalently a pro-linear map $\Q^1 \to \Q^{\pi_0 S}$ --- in 
either way interpreted as an element in $\Q^{\pi_0 S}$.
In the first viewpoint, the linear map is
\begin{eqnarray*}
  \Q_{\pi_0 S} & \longrightarrow & \Q_1  \\
  \delta_s & \longmapsto & \int^1 \norm{F_s} \delta_1 = \norm{F_s} \delta_1
\end{eqnarray*}
which is precisely the function
\begin{eqnarray*}
  \pi_0 S & \longrightarrow & \Q  \\
  s & \longmapsto & \norm{f(s)} .
\end{eqnarray*}
In the second viewpoint, it is the prolinear map
\begin{eqnarray*}
  \Q^1 & \longrightarrow & \Q^{\pi_0 S}  \\
  \delta_1 & \longmapsto & \sum_s \norm{F_s} \delta^s
\end{eqnarray*}
which of course also is the function $s \mapsto \norm{f(s)}$.

In conclusion:
\end{blanko}

\begin{prop}\label{prop:card-prsh-ptwise}
  The cardinality of a presheaf $f: S \to \grpd$ is computed pointwise:
  $\norm f$ is the function
  \begin{eqnarray*}
    \pi_0 S & \longrightarrow & \Q  \\
    s & \longmapsto & \norm{f(s)} .
  \end{eqnarray*}
  In other words, it is obtained by postcomposing with the basic
  homotopy cardinality.
\end{prop}

\begin{blanko}{Example.}\label{card(h)}
  The cardinality of the representable functor $h^t : S \to \grpd$ is
  \begin{eqnarray*}
    \pi_0 S & \longrightarrow & \Q  \\
    s & \longmapsto & \norm{\Map(t,s)} = \begin{cases}
      \norm{\Omega(S,s)} & \text{ if } t\simeq s \\
      0 & \text{ else.}
    \end{cases}
  \end{eqnarray*}
\end{blanko}

\begin{blanko}{Remark.}
    Note that under the finite fundamental equivalence
$\grpd^S \simeq \Grpd_{/S}^{\relfin}$ (\ref{fundamentalfinite}),
the representable
presheaf $h^s$ corresponds to $\name s$, the name of $s$,
which happens to belong also to the subcategory $\grpd_{/S} \subset
\Grpd_{/S}^{\relfin}$, but that the cardinality of $h^{s}\in \grpd^S$ 
{\em cannot} be identified with
the cardinality of $\name s \in \grpd_{/S}$.
This may seem confusing at first, but it is forced upon us
by the choice of normalisation of the functor
$$
\normnorm{ \ \ } : \ind\lin \to \Vect
$$
which in turn looks natural since the extra factor $\norm{ 
T_{[t]}}$ comes from an integral.  A further feature of this
apparent discrepancy is the following.
\end{blanko}

\begin{prop}
  Cardinality of the canonical perfect pairing at the $\infty$-groupoid level 
  (\ref{duality}) yields
  precisely the perfect pairing on the vector-space level.
\end{prop}

\begin{proof}
  We take cardinality of the perfect pairing
\begin{eqnarray*}
    \grpd_{/S} \times \grpd^S  &\longrightarrow& \grpd \\
    (p , f) &\longmapsto& f(p) \\
    (\name s, h^t) &\longmapsto& \begin{cases}
    \Omega(S,s) & \text{ if } t\simeq s \\
    0 & \text{ else .}
\end{cases}
\end{eqnarray*}
Since the cardinality of $\name s$ is $\delta_s$, while the 
cardinality of $h^t$ is $\norm{\Omega(S,t)}\delta^t$, 
the cardinality of the pairing becomes
$$
(\delta_s, \norm{ \Omega(S,t)}\delta^t) \longmapsto 
\begin{cases}
    \norm{\Omega(S,t)} & \text{ if } t \simeq s \\
    0 & \text{ else },
\end{cases}
$$
or equivalently:
$$
(\delta_s, \delta^t) \longmapsto 
\begin{cases}
    1 & \text{ if } t \simeq s \\
    0 & \text{ else },
\end{cases}
$$
as required.
\end{proof}

\begin{blanko}{Remarks.}
  The definition of meta cardinality involves a convention, namely for a span 
  $S \leftarrow M \to T$ to include
  the factor $\norm{T_{[t]}}$.  In fact, as observed by
  Baez--Hoffnung--Walker~\cite{Baez-Hoffnung-Walker:0908.4305}, other
  conventions are possible: for any real numbers
  $\alpha_1$ and $\alpha_2$
  with $\alpha_1+\alpha_2=1$, it is possible to use the factor
  $$
  \norm{S_{[s]}}^{\alpha_1} \norm{T_{[t]}}^{\alpha_2} .
  $$
  They choose to use $0+1$ in some cases and $1+0$ in 
  other cases, according to what seems more practical.  We think
  that these choices can be explained by the side of duality on which the
  constructions take place.

Our convention with the $\norm{T_{[t]}}$ normalisation yields the
`correct' numbers in all the applications of the theory that motivated
us, as exemplified below.
\end{blanko}

\begin{blanko}{Incidence coalgebras and incidence algebras of decomposition spaces.}
  A main motivation for us is the theory of decomposition spaces 
  \cite{GKT:DSIAMI-1}, \cite{GKT:DSIAMI-2}, \cite{GKT:MI}.
  A {\em decomposition space} is a simplicial $\infty$-groupoid 
  $X: \simplexcategory\op\to\Grpd$ satisfying
  an exactness condition precisely so as to make the following comultiplication
  law coassociative, up to coherent homotopy.
  The natural span
  $$
\xymatrix{
 X_1  & \ar[l]_{d_1}  X_2\ar[r]^-{(d_2,d_0)} &  X_1\times X_1 
}
$$
defines a linear functor, the {\em comultiplication}
\begin{eqnarray*}
  \Delta : \Grpd_{/ X_1} & \longrightarrow & 
  \Grpd_{/( X_1\times X_1)}  \\
  (S\stackrel s\to X_1) & \longmapsto & (d_2,d_0) \lowershriek \circ d_1 \upperstar(s) .
\end{eqnarray*}
(and similarly the span $X_1 \stackrel{s_0}\leftarrow X_0 \to 1$ defines the counit).
This is called the incidence coalgebra of $X$.  If the maps
$X_0 \stackrel{s_0}\longrightarrow X_1 \stackrel{d_1}\longleftarrow X_2$
are both finite 
and $X_1$ is locally finite 
then this coalgebra structure 
restricts to a coalgebra structure on $\grpd_{/X_1}$, which in turn
descends to
$\Q_{\pi_0 X_1}$ under taking cardinality \cite{GKT:DSIAMI-2}.

An example is given by the fat nerve of $\mathbf V$, the category of
finite-dimensional vector spaces over a finite field and linear injections.
D\"ur~\cite{Dur:1986} obtained the $q$-binomial coalgebra from this
example by a reduction
step, identifying two linear injections if their cokernels have the same
dimension.  The coalgebra can also be obtained directly from a decomposition
space, namely the Waldhausen $S$-construction on $\mathbf V$.
We check in \cite{GKT:ex} that the cardinality of this comultiplication
gives precisely the classical Hall numbers (with the present convention).
\end{blanko}

\begin{blanko}{Zeta functions.}
  For $X$ a decomposition space with $X_0 \stackrel{s_0}\longrightarrow X_1
\stackrel{d_1}\longleftarrow X_2$ both finite maps, the dual space of $\grpd_{/X_1}$
is $\grpd^{X_1}$, underlying the incidence {\em algebra}.  Its multiplication
is given by a convolution formula.
In here there is a canonical element, the
`constant' linear functor given by the span $X_1 \stackrel=\leftarrow X_1 \to 1$
(corresponding to the terminal presheaf), which is called the {\em zeta functor}
\cite{GKT:DSIAMI-2}.  By \ref{prop:card-prsh-ptwise}, the cardinality of the
terminal presheaf is the constant function $1$.  Hence the cardinality of the
zeta  functor is the classical zeta function in incidence algebras.
\end{blanko}

\begin{blanko}{Green functions.}
  The zeta function is the `sum of everything', with no symmetry factors.  A `sum
of everything', but {\em with} symmetry factors, appeared in our work
\cite{GalvezCarrillo-Kock-Tonks:1207.6404} on the Fa\`a di Bruno and Connes--Kreimer
bialgebras, namely in the form of combinatorial Green functions (see also
\cite{Kock:1512.03027}).

The coalgebra in question is then the completion of the finite incidence algebra
$\Grpd_{/X_1}^{\relfin}$, where $X_1$ is the groupoid of forests (or
more precisely, $P$-forests for $P$ a polynomial 
functor~\cite{Gambino-Kock:0906.4931}, \cite{Kock:0807}).
Of course we know that $\Grpd_{/X_1}^{\relfin}$ is canonically
equivalent to $\grpd^{X_1}$, but it is important here to keep track of
which side of duality we are on.  The Green function, which is in reality a 
distribution rather than a function,
lives on the coalgebra side, and more precisely in the completion.
(The fact that the comultiplication extends to the completion
is due to the fact that not only $d_1: X_2 \to X_1$ is finite,
but that also $X_2 \to X_1\times X_1$ is finite (a feature common
to all Segal $1$-groupoids
with $X_0$ locally finite).)

Our Green function, shown to satisfy the
Fa\`a de Bruno formula in $\Grpd_{/X_1}^{\relfin}$, is
$T \to X_1$, the full inclusion of the groupoid of $P$-trees $T$
into the groupoid of $P$-forests.
Upon taking cardinality, with the present conventions, we obtain
precisely the series
$$
G = \sum_{t\in \pi_0 T} \frac{\delta_t}{\norm{\Aut(t)}} ,
$$
the sum of all trees weighted by symmetry factors,
which is the usual combinatorial Green function in Quantum Field 
Theory, modulo the difference between trees and graphs~\cite{Kock:1512.03027}.
The important symmetry factors appear correctly because we are on the coalgebra
side of the duality.
\end{blanko}

\end{document}